\documentclass[11pt]{article}  

\usepackage{multicol}

\usepackage{amsmath}
\usepackage{latexsym,amssymb}
\usepackage[mathscr]{eucal}

\usepackage{color}
\usepackage{enumitem}

\usepackage{hyperref}
\newtheorem{thm}{Theorem}[section]
\newtheorem{lem}[thm]{Lemma}	
\newtheorem{pro}[thm]{Proposition}
\newtheorem{cor}[thm]{Corollary}
\newtheorem{dfn}[thm]{Definition}
\newtheorem{eg}[thm]{Example}
\newtheorem{result}[thm]{Result}

\newcommand{\bea}{\begin{eqnarray*}}
\newcommand{\eea}{\end{eqnarray*}}

\newcommand{\ben}{\begin{enumerate}}
\newcommand{\een}{\end{enumerate}}

\newcommand{\bi}{\begin{itemize}}
\newcommand{\ei}{\end{itemize}}

\newenvironment{proof}{\noindent \textbf{Proof.}\hspace{.7em}}
                   {\hfill $\Box$                    \vspace{10pt}}

\newcommand{\mc}{\mathcal}

\newcommand{\Z}{\mathbb Z}

\newcommand{\R}{\mathcal R}

\newcommand{\Q}{\mathcal Q}
\newcommand{\C}{\mathcal C}
\newcommand{\I}{\mathcal I}

\newcommand{\pr}{^{\prime}}
\newcommand{\Su}{\mc S}
\newcommand{\IS}{\mc IS}
\newcommand{\Pc}{\mc P}
\newcommand{\D}{\mc D}
\newcommand{\M}{\mc M}

\DeclareMathOperator{\Dom}{Dom}
\DeclareMathOperator{\Img}{Im}

\begin{document}

\title{Constellations with range and IS-categories}
\author{Victoria Gould and Tim Stokes}
\date{}

\maketitle

\begin{abstract}
Constellations are asymmetric generalisations of categories.  Although they are not required to possess a notion of range, many natural examples do. These include  commonly occurring constellations  related to concrete categories (since they model surjective morphisms), and also others arising from quite different sources, including from well-studied classes of semigroups. We show how constellations with a well-behaved range operation are nothing but ordered categories with restrictions.  We characterise abstractly those categories that are canonical extensions of constellations with range, as so-called IS-categories.  Such categories contain distinguished subcategories of insertions (which are monomorphisms) and surjections (in general different to the epimorphisms) such that each morphism admits a unique factorisation into a surjection followed by an insertion. Most familiar concrete categories are IS-categories, and we show how some of the well-known properties of these categories arise from the fact that they are IS-categories.  For appropriate choices of morphisms in each, the category of IS-categories is shown to be equivalent to the category of constellations with range.  
\end{abstract}

\medskip

\noindent{{\bf 2020 Mathematics Subject Classification:} Primary: 18A32, 18B99, Secondary: 20M75

\section{Introduction}\label{sec:intro}

Constellations are ``one-sided" versions of categories, in which there is a notion of domain but in general no notion of range.  They were first defined by Gould and Hollings in \cite{constell}, where the purpose was to obtain a variant of the so-called Ehresmann-Schein-Nambooripad (ESN) Theorem relating inverse semigroups to inductive groupoids (see \cite{Ehres} and \cite{scheinESN}) which would apply to left restriction semigroups.  

Constellations were subsequently studied further for their own sake in \cite{constgen}, where  the assumption that constellations were small (set-based) was dropped, something we do not require in the current work either.   It was shown that, given a constellation $\Q$, a category $C(\Q)$ can be obtained by means of a straightforward extension process that can be viewed as a one-sided variant of idempotent completion.  For example, the category of sets may be obtained via this extension process from the constellation of sets, the latter consisting of surjective functions with composition of two functions defined if and only if the image of the first is contained in the domain of the second.  

In fact, it is typical that a concrete category $\C$ can be viewed in this way as a canonical extension $C(\Q)$ of the constellation $\Q$ consisting of the surjective morphisms.  However, the surjective morphisms in these concrete categories generally form a subcategory (equal to, or at least contained in, the subcategory of epimorphisms), and hence come equipped with a notion of range.  When viewed as a constellation with such a range operation, a number of properties are typically satisfied.  In this paper, we consider such constellations with range, how they relate to categories in general, and how they may be used to enhance category theory itself by exploiting the canonical extension concept. 

Following some preliminaries in Section \ref{prelims}, we turn our attention to the notion of a constellation with range in Section \ref{secranconst}, and consider examples and non-examples in Section \ref{examples}.  In Section \ref{correspsec}, it is shown that constellations with range may  be equivalently viewed as ordered categories with restriction.  We then turn our attention for the remainder to canonical extensions of constellations with range.  The basic idea is reviewed in Section \ref{canonext}, and some consequences of assuming that $\Q$ has range are collated.  Then in Section \ref{IScat}, the concept of an IS-category is defined and developed as an abstraction of $C(\Q)$ where $\Q$ is a constellation with range; these have some formal similarity to categories equipped with a factorization system.  In Section \ref{apply}, some applications are given, showing how working in IS-categories can help explain the behaviour of many of the familiar concrete categories of mathematics, with the notions of epimorphisms, subobjects, equalisers and factorization systems all considered. In Section \ref{corresp2}, it is shown that IS-categories are up to isomorphism exactly canonical extensions of constellations with range, and a suitable categorical equivalence is obtained between constellations with range and IS-categories.  Finally, some open questions are considered in Section \ref{open}.

\section{Algebraic Preliminaries}  \label{prelims}

Throughout, we generally write functions on the right of their arguments rather than the left, so ``$xf$" rather than ``$f(x)$".  Correspondingly, we write function compositions left to right, so that ``$fg$" is ``first $f$, then $g$", rather than the other way around.  An exception to this is unary operation application; if $D$ is a unary operation on a set $S$, we write $D(s)$ for $s\in S$ rather than $sD$.

Let $\C$ be a class (usually a set) with a partial binary operation.  Recall that $e\in \C$ is a \emph{right identity} if it is such that, for all $x\in \C$, if $x\cdot e$ is defined then it equals $x$; left identities are defined dually.  An \emph{identity} is both a left and right identity.  (Note we are not assuming that $e\cdot e$ exists in any of these cases.)

Following \cite{constgen} and \cite{lawson1} where an object-free formulation was also used, recall that a \emph{category} $(\C,\circ)$, often just denoted $\C$, is a class $\C$ with a partial binary operation $\circ$ satisfying the following:
\ben[label=\textup{(C\arabic*)}]
\item $x\circ(y\circ z)$ exists if and only if $(x\circ y)\circ z$ exists, and then the two are equal;
\item if $x\circ y$ and $y\circ z$ exist then so does $x\circ(y\circ z)$;
\item for each $x\in \C$, there are identities $e,f$ such that $e\circ x$ and $x\circ f$ exist.
\een

The identities $e,f$ in (C3) are easily seen to be unique,  
and  we write $D(x)=e$ and $R(x)=f$.  (Note that this is the opposite of the convention often used, and corresponds to the fact that we view a composition of functions $fg$ as ``first $f$, then $g$", discussed earlier.)

It also follows easily that for every identity $e$ we have $e\circ e=e$ and $D(e)=e=R(e)$.  Moreover, the collection of domain elements $D(x)$ (equivalently, range elements $R(x)$) is precisely the collection of identities in the category.  We will often view a category as a partial algebra $(\C,\circ,D,R)$.  If $x,y$ are elements of a category $\C$, then the product $x\circ y$ exists if and only if $R(x)=D(y)$, and if $x\circ y$ exists then $D(x\circ y)=D(x)$ and $R(x\circ y)=R(y)$.

 Let $\C$ be a category.  As usual, we say that $s\in \C$ is an {\em epimorphism} if for all $x,y\in \C$, if $s\circ x=s\circ y$ then $x=y$, and that it is a  {\em monomorphism} if  for all $x,y\in \C$, if $x\circ s=y\circ s$ then $x=y$.  (Recall that we are reading compositions left to right.)

Following an earlier equivalent definition given in \cite{constell}, we make use of a result in \cite{constgen} and define a \emph{constellation}  to be a class $\Q$ equipped with a partial binary operation $\cdot$ satisfying the following:
\ben[label=\textup{(Q\arabic*)}]
\item if $x\cdot(y\cdot z)$ exists then $(x\cdot y)\cdot z$ exists, and then the two are equal;
\item if $x\cdot y$ and $y\cdot z$ exist then so does $x\cdot(y\cdot z)$;
\item for each $x\in P$, there is a unique right identity $e$ such that $e\cdot x=x$.
\een
Since $e$ in (Q3) is unique given $x\in \Q$, we call it $D(x)$.  
It follows that $D(\Q)=\{D(s)\mid s\in \Q\}$ is the set of right identities of $\Q$, a set we call the {\em projections} of $\Q$.  We adopt the usual convention of referring to the constellation $(\Q,\cdot)$ simply as $\Q$ if there is no ambiguity.  However, because $D$ can be viewed as a unary operation, we also often view constellations as partial algebras $(\Q,\cdot\, ,D)$.  As shown in \cite{constgen}, every category becomes a constellation when the operation $R$ is ignored.  We say a constellation $(\Q,\cdot)$ is {\em small} if $\Q$ is a set.

The following are some useful basic facts and definitions about constellations, to be found in \cite{constell} or \cite{constgen}.

\begin{result}  \label{2p3}
For elements $s,t$ of the constellation $\Q$, $s\cdot t$ exists if and only if $s\cdot D(t)$ exists, and then $D(s\cdot t)=D(s)$.
\end{result}

A constellation is {\em categorial} if it arises from a category as a reduct (obtained by dropping $R$). 

\begin{result}  \label{catunique}
Let $\Q$ be a constellation.  Then $\Q$ is categorial if and only if for all $s\in \Q$ there is a unique $e\in D(\Q)$ such that $s\cdot e$ exists, and then $R(s)=e$ when $\Q$ is viewed as a category.
\end{result}

The constellation $\Q$ is {\em normal} if $\mbox{for all }e,f\in D(\Q)$, \mbox{ if $e\cdot f$ and $f\cdot e$ exist, then } $e=f$.  Our next result highlights the fact that constellations are intimately associated with orders.

\begin{result}  \label{parnormal}
If $\Q$ is a constellation, define the relation $\leq$ on $D(\Q)$ by $e\leq f$ if and only if $e\cdot f$ exists.  Then $\leq$ is a quasiorder we call the {\em standard quasiorder} on $D(\Q)$.  Defining $s\leq t$ for $s,t\in \Q$ whenever $s=e\cdot t$ for some $e\in D(\Q)$ (equivalently, $s=D(s)\cdot t$) makes $\leq$ a quasiorder on all of $\Q$ that agrees with the standard quasiorder on $D(\Q)$, called the {\em natural quasiorder} on $\Q$.  In both these cases, the quasiorder is a partial order if and only if $\Q$ is normal, and then we use ``order" rather than ``quasiorder".
\end{result}

An important example of a small normal constellation, introduced in \cite{constell}, is ${\mc C}_X$, consisting of partial functions on the set $X$, in which $s\cdot t$ is the usual composite of $s$ followed by $t$ provided $\Img(s)\subseteq \Dom(t)$, and undefined otherwise, and $D(s)$ is the restriction of the identity map on $X$ to $\Dom(s)$.  

A {\em subconstellation} $Q$ of a constellation $\Q$ is a subset of $\Q$ that is closed under the constellation product wherever it is defined, and closed under $D$; $Q$ is then a constellation in its own right as shown in \cite{constgen}, where it was also shown that every small normal constellation embeds as a subconstellation in the (normal) constellation $\C_X$ for some choice of $X$.

For constellations, the notion of morphism is as follows.  If $\Q_1,\Q_2$ are constellations, a function $\rho:\Q_1\rightarrow \Q_2$ is a {\em radiant} if for all $s,t\in \Q_1$ we have $D(s\rho)=D(s)\rho$, and if  $s\cdot t$ exists then so does $(s\rho)\cdot (t\rho)$ and indeed $(s\cdot t)\rho=(s\rho)\cdot (t\rho)$.  As observed in \cite{constell}, the class of constellations is a category in which the morphisms are radiants.  

As in \cite{constgen}, we say that a radiant $\rho\colon \mathcal{P}\rightarrow \Q$ is \emph{full} if for all $s,t\in \mathcal{P}$ for which $(s\rho)\cdot (t\rho)$ exists and is in the image of $\rho$, there are $s\pr,t\pr\in \mathcal{P}$ such that $s\rho=s\pr\rho$, $t\rho=t\pr\rho$ and $s\pr\cdot t\pr$ exists in $\mathcal{P}$.  

There is also a notion of congruence for constellations. 

\begin{dfn}  \label{congdef}
Let $\delta$ be an equivalence relation on the constellation $\Q$.  
As in \cite{constgen}, $\delta$ is a \emph{congruence} if, whenever $(s_1,s_2)\in \delta$ then $(D(s_1),D(s_2))\in \delta$ and if also $(t_1,t_2)\in \delta$, and both $s_1\cdot t_1$ and $s_2\cdot t_2$ are defined, then $(s_1\cdot t_1, s_2\cdot t_2)\in \delta$; it is a \emph{strong congruence} if it also has the property that $s_1\cdot t_1$ is defined if and only if $s_2\cdot t_2$ is defined, for all $(s_1,t_1),(s_2,t_2)\in \delta$. 
\end{dfn}

If $\Q_1,\Q_2$ are constellations, the kernel $\delta$ of the radiant $\rho\colon \Q_1\rightarrow \Q_2$ (defined to be an equivalence relation on $\Q_1$ in the usual way) is a congruence but need not be strong, as noted in \cite{constgen}.

The following comes from Propositions $2.27$ and $2.28$ in \cite{constgen}.

\begin{result}  \label{canonmap}
If $\C$ is a category with $\delta$ a constellation congruence on it such that $\C/\delta$ is a constellation, then the canonical map $\rho:\C\rightarrow \C/\delta$ is a full surjective radiant.  Conversely, if $\Q$ is a constellation and $\rho:\C\rightarrow \Q$ is a full surjective radiant, then $\C/\ker(\rho)\cong \Q$.  
\end{result}

So the congruences on a category $\C$ that give constellation quotients are precisely the kernels of full surjective radiants mapping from $\C$ to a constellation.

In \cite{constgen}, the notion of the canonical extension of a constellation was developed.  Given a constellation $\Q$, it is the category constructed as follows:
 $$C(\Q)=\{(s,e)\mid s\in \Q, e\in D(\Q),\exists s\cdot e\},$$
where we define  \[D((s,e))=(D(s),D(s))\mbox{ and }R((s,e))=(e,e),\]
and where $R((s,e))=D((t,f))$,  (that is, $e=D(t)$), we set \[(s,e)\circ (t,f)=(s\cdot t,f).\] That $(C(\Q),\circ\, ,D,R)$ is a  category is shown in \cite{constgen}.

A constellation $\Q$ is {\em composable} if for every $a\in \Q$ there exists $b\in \Q$ (equivalently, there exists $b\in D(\Q)$) such that $a\cdot b$ exists. 

The following is Proposition 3.2 of \cite{constgen}.

\begin{result} \label{surad}
If $\Q$ is a composable constellation, the mapping $\rho: C(\Q)\rightarrow \Q$ given by $(s,e)\rho=s$ is a full surjective radiant.
\end{result}

Associated to the canonical extension construction is the notion of a {\em canonical congruence} $\delta$ on a constellation $\Q$: it is defined in \cite{constgen} to be a congruence as in Definition \ref{congdef} additionally satisfying the following: 
\bi
\item $\delta$ separates projections: for all $e,f\in D(\Q)$, if $(e,f)\in\delta$ then $e=f$;
\item if $(a,b)\in \delta$ and $a\cdot e$ and $b\cdot e$ both exist for some $e\in D(\mathcal{P})$, then $a=b$.
\ei

It was shown in \cite[Corollary 3.7]{constgen}  that if $\Q$ is a category, then an equivalence relation $\delta$ on the constellation reduct of $\Q$ is a canonical congruence if and only if 
\bi
\item if $(a,b)\in \delta$ then $D(a)=D(b)$;
\item if $(a,b)\in \delta$ and $R(a)=R(b)$, then $a=b$;
\item if $(b,c)\in \delta$ and $a\cdot b$ (and hence $a\cdot c$) exists, then $(a\cdot b,a\cdot c)\in\delta$. 
\ei

The next result is a combination of Proposition $3.9$ and Theorem $3.10$ in \cite{constgen}.

\begin{result}  \label{factext}
If $\mathcal{K}$ is a category with $\delta$ a canonical congruence on it, then $\mathcal{K}/\delta$ is a constellation, and $\mathcal{K}\cong C(\mathcal{K}/\delta)$, via an isomorphism given by $s\mapsto ([s],[R(s)])$, where $[x]$ denotes the $\delta$-class containing $s\in \mathcal{K}$.
\end{result}

\section{Range in a constellation}  \label{secranconst}

In the examples of constellations considered in \cite{constgen}, most had a notion of range, which in all cases had the property that the range $R(s)$ of an element $s$ was an element of $D(\Q)$ with the property that $s\cdot R(s)$ exists, and $R(s)$ was the smallest $e\in D(\Q)$ with respect to the standard quasiorder on $D(\Q)$ for which $s\cdot e$ exists.  For example, for the case of ${\mc C}_X$, $R(s)$ is the identity map on $X$ restricted to the image of $s$. In such cases, $\Q$ is such that for all $s\in \Q$, the set 
$$s_D=\{e\in D(\Q)\mid s\cdot e\mbox{ exists, and for all }f\in D(\Q), s\cdot f\mbox{ exists implies } e\leq f\}$$ contains a single element.  (In a general constellation, $s_D$ could be empty, or even have more than one element if the standard quasiorder on $D(\Q)$ is not a partial order, depending on the choice of $s$.)

\begin{dfn}
A {\em constellation with pre-range} is a constellation $\Q$ in which for all $s\in \Q$ the set $s_D$ has a single element (which we call $R(s)$).
\end{dfn}  

Again, we often view a constellation with pre-range as a partial binary algebra having two unary operations, $D$ and $R$. 
The next lemma follows immediately from the very definition of $\leq$ in a constellation.

\begin{pro}  \label{Rlaws}
Suppose $\Q$ is a constellation with pre-range.  Then 
\ben[label=\textup{(R\arabic*)}]
\item $D(R(s))=R(s)$ for all $s\in \Q$;
\item $s\cdot R(s)$ exists for all $s\in \Q$;
\item if $s\cdot t$ exists then so does $R(s)\cdot t$, for all $s,t\in \Q$;
\item $\Q$ is normal.
\een
\end{pro}
\begin{proof} 
(R1) and (R2) are immediate.   If $s\cdot t$ exists then $s\cdot D(t)$ exists by Result \ref{2p3}, so $R(s)\leq D(t)$, and so $R(s)\cdot D(t)$ exists, so $R(s)\cdot t$ exists, again by Result \ref{2p3}, which proves (R3).  
Finally, if $e,f\in D(\Q)$ with $e\cdot f$ and $f\cdot e$ both existing, then $e\leq f\leq e$ and $e,f\in  e_D$ so that $e=f$ and (R4) holds.
\end{proof}

Conversely, we have the following.

\begin{pro}  \label{Rmin}
Suppose $(\Q,\cdot\, ,D)$ is a constellation with additional unary operation $R$ satisfying the laws (R1)--(R4) in Proposition \ref{Rlaws}.  Then $(\Q,\cdot\, ,D,R)$ is a constellation with pre-range.
\end{pro}
\begin{proof}
For all $s\in \Q$, $R(s)\in D(\Q)$ by (R1).  (R2) shows that $R(s)$ is one $e\in D(\Q)$ for which $s\cdot e$ exists, while if also $s\cdot f$ exists for some $f\in D(\Q)$, then $R(s)\cdot f$ exists by (R3), so $R(s)\leq f$ under the standard order on $D(\Q)$, which is a partial order by (R4) and Result \ref{parnormal}.  So $R(s)$ is the (unique) smallest $e\in D(S)$ for which $s\cdot e$ exists.
\end{proof}
 
There is a kind of dual version of Result \ref{2p3} applying to constellations with pre-range.

\begin{lem} \label{Rst}
Suppose $\Q$ is a constellation with pre-range.  Then for all $s,t\in \Q$, $s\cdot t$ exists if and only if $R(s)\cdot t$ exists.  Hence for all $s,t\in \Q$,  if $s\cdot t$ exists, then $R(s\cdot t)\leq R(t)$.
\end{lem}
\begin{proof}
(R3) asserts one direction.  Conversely, suppose that $R(s)\cdot t$ exists.  Now $s\cdot R(s)$ exists by (R2), so it follows from (Q2) that $s\cdot (R(s)\cdot t)$ exists and hence by (Q1) so does $(s\cdot R(s))\cdot t=s\cdot t$.  Hence if $s\cdot t$ exists, then it equals $s\cdot (t\cdot R(t))=(s\cdot t)\cdot R(t)$, giving that $R(s\cdot t)\leq R(t)$.
\end{proof}

Next we explain why the term ``pre-range"  rather than just ``range" has been used thus far. Note that $D(s\cdot D(t))=D(s\cdot t)$ ($=D(s)$) whenever $s\cdot t$ exists.  However, in a constellation with pre-range, the corresponding condition that $R(s\cdot t)=R(R(s)\cdot t)$ whenever $s\cdot t$ exists need not hold.  

\begin{dfn}
The constellation with pre-range $(\Q,\cdot\, ,D,R)$ is a {\em constellation with range} if it satisfies {\em the right congruence condition}: $R(R(s)\cdot t)=R(s\cdot t)$ whenever $s\cdot t$ exists (which is whenever $R(s)\cdot t$ exists by Lemma \ref{Rst}).  
\end{dfn}

The name here derives from the fact that in a constellation with pre-range $\Q$, the equivalence relation $\theta$ given by $s\, {\theta}\, t$ if and only if $R(s)=R(t)$ is a right congruence (suitably defined) if and only if $\Q$ is a constellation with range.  We omit the details.

It follows easily that in a constellation with range, $R(s\cdot t)\leq R(t)$, with equality if $R(s)=D(t)$, since in that case
$R(s\cdot t)=R(R(s)\cdot t)=R(D(t)\cdot t)=R(t)$.

\begin{pro}  \label{CX} The constellation
$\C_X$ is a constellation with range, with $R(s)$ the restriction of the identity function to the range of $s\in \C_X$.
\end{pro}
\begin{proof}
We already noted that defining $R$ as above clearly makes $\C_X$ a constellation with pre-range.

Suppose $s,t\in \C_X$.  If $s\cdot t$ exists in $\C_X$, then $R(s)\cdot t$ exists by (R3), and in this case, $R(s\cdot t)=R(st)=R(R(s)t)=R(R(s)\cdot t)$ as computed in the bi-unary semigroup of partial functions on $X$ equipped with domain and range, $(\mathcal{PT}_X,\times,D,R)$ (where we omit $\times$ in the products given).  This is because the law $R(st)=R(R(s)t)$ is known to hold for $\mathcal{PT}_X$  (see \cite{ss4} for example).  
\end{proof}

As shown in \cite{constgen}, every normal constellation embeds in $\C_X$ for suitable $X$, hence in a constellation with range. 

Before leaving the example of $\C_X$, we note that every {\em left restriction semigroup} $S$ embeds into some $\mathcal{PT}_X$, and moving from the semigroup $S$ to the corresponding sub-constellation $\mathcal{P}(S)$ of $\C_X$, together with the natural order, allows us to recover $S$ \cite{constell}. Regarding $S$ as a unary subsemigroup of  $\mathcal{PT}_X$ with unary operation of domain, in certain cases we may make $\mathcal{P}(S)$  into a constellation with pre-range by defining $R(s)$ to be the identity map restricted to the subset
\[\bigcap\{ \mbox{dom\,} t: \exists s\cdot t\}.\]
In the case $S$ is finite and composable when viewed as a constellation, or  $S=\mathcal{PT}_X$, or more generally $S$ contains all identity maps on subsets of $X$, we have that $\mathcal{P}(S)$ is a constellation with pre-range. In the latter two cases, $\mathcal{P}(S)$ is a constellation with range and $S$ is both left restriction and right Ehresmann; these semigroups (in fact, they are properly speaking bi-unary semigroups), are the topic of the very recent investigation \cite{margolis} of Margolis and Stein.

Clearly, every category $(\C,\circ,D,R)$ is a constellation with range.  Categories are easy to characterise within constellations with range, thanks to Result \ref{catunique}: a constellation with range is a category if and only if, for $s\in \Q$, if $s\cdot e$ exists then $e=R(s)$.  Indeed categories provide further motivation for including the congruence condition in the definition of constellations with range.  In the example $(\C_X,\cdot\, ,D,R)$ of a constellation with range, the usual category product of elements in $\C_X$ may be defined as follows: if $s,t\in \C_X$, define
\[
s\circ t:=\begin{cases} s\cdot t\mbox{ if }R(s)=D(t),\\
\mbox{undefined otherwise.}
\end{cases}
\]
This seems a natural way to seek to obtain a category from any constellation with pre-range $\Q$.  In fact this is possible if and only if $\Q$ is a constellation with range.

\begin{pro}  \label{catran}
If $(\Q,\cdot\, ,D,R)$ is a constellation with pre-range, then $(\Q,\circ\, ,D,R)$ is a category under the new partial operation $$s\circ t=s\cdot t\mbox{ if }R(s)=D(t),\mbox{ and undefined otherwise}$$ 
if and only if $(\Q,\cdot\, ,D,R)$ is a constellation with range.  
\end{pro}
\begin{proof}
Suppose $(\Q,\cdot\, ,D,R)$ is a constellation with range, and define $s\circ t$ as above; this makes sense because if $R(s)=D(t)$ then $R(s)\cdot t=D(t)\cdot t$ exists, and so $s\cdot t$ exists by Lemma \ref{Rst}.  Also,  if $x\circ y$ exists then of course $D(x\circ y)=D(x)$, but also $R(x)=D(y)$ and so by the comment preceding Proposition~\ref{CX} we  have
$R(x\circ y)=R(y)$.

Hence for $x,y,z\in \Q$, the following are equivalent: $(x\circ y)\circ z$ exists; $R(x)=D(y)$ and $R(x\circ y)=D(z)$; $R(x)=D(y)$ and $R(y)=D(z)$; $x\circ y$ and $y\circ z$ exist; $R(y)=D(z)$ and $R(x)=D(y\circ z)$; $x\circ(y\circ z)$ exists.  When both are defined, we have
$$x\circ(y\circ z)=x\cdot(y\cdot z)=(x\cdot y)\cdot z=(x\circ y)\circ z.$$
All of this establishes (C1).  This argument also shows that if $x\circ y$ and $y\circ z$ are defined then $x\circ (y\circ z)$ exists, which establishes (C2).

Suppose $e\in D(\Q)$, with $s\in \Q$.  If $s\circ e$ exists, then it equals $s\cdot e=s$.  Also, $e\circ s$ exists if and only if $e=R(e)=D(s)$, and then $e\circ s=e\cdot s=D(s)\cdot s=s$, so $e$ is both a right and left identity under $\circ$, hence an identity, and of course $e\circ s$ exists for $e=D(s)$.  Moreover $s\circ R(s)$ exists since $R(s)=D(R(s))$.  So (C3) holds.

Conversely, suppose $(\Q,\cdot\, ,D,R)$ is a constellation with pre-range for which the above definition makes it a category.  Suppose $x\cdot y$ exists.  Then $z=R(x)\cdot y$ exists by (R3), and clearly $R(x)=D(z)$, so $x\circ z$ exists, and 
$$R(x\cdot y)=R((x\cdot R(x))\cdot y)=R(x\cdot (R(x)\cdot y))=R(x\cdot z)=R(x\circ z)=R(z)=R(R(x)\cdot y).$$ 
This shows that $(\Q,\cdot\, ,D,R)$ satisfies the congruence condition and hence is a constellation with range.
\end{proof}

\begin{dfn}  \label{derived}
We call the category $(\Q,\circ\, ,D,R)$ obtained from the constellation with range $(\Q,\cdot\, ,D,R)$ as in the previous result its {\em derived category}.  
\end{dfn}

\begin{dfn}
An element $a$ of a constellation with range $\Q$ is {\em left cancellative} (or an {\em epimorphism} if $\Q$ is a category) if whenever $a\cdot b=a\cdot c$, we have $R(a)\cdot b=R(a)\cdot c$.  Moreover $\Q$ is {\em left cancellative} if every element is left cancellative.
\end{dfn}

The above definition coincides with the usual definition of  ``epimorphism" in a category $\Q$.   We will also use the term ``monomorphism" for categories in the usual way.
Every constellation with range consisting of partial functions is left cancellative. 

\begin{pro}  \label{oneway}
A constellation with range is left cancellative if and only if its derived category is.
\end{pro}

\begin{proof}
If $(\Q,\cdot\, ,D,R)$ is left cancellative, and $x\circ y=x\circ z$, then $x\cdot y=x\cdot z$ and $R(x)=D(y)=D(z)$, so $y=D(y)\cdot y=R(x)\cdot y=R(x)\cdot z=D(z)\cdot z=z$.  So the derived category $(\Q,\circ\, ,D,R)$ is left cancellative.  Conversely, suppose the derived category is left cancellative.  If $x\cdot y=x\cdot z$ then $R(x)\cdot y$ and $R(x)\cdot z$ exist by (R3), and so $x\cdot(R(x)\cdot y)=(x\cdot R(x))\cdot y=x\cdot y=x\cdot z=(x\cdot R(x))\cdot z=x\cdot(R(x)\cdot z)$. Since  $D(R(x)\cdot y)=D(R(x)\cdot z)=D(R(x))=R(x)$, the products $x\circ (R(x)\cdot y)$ and $x\circ (R(x)\cdot z)$ exist are equal, giving $R(x)\cdot y=R(x)\cdot z$.  Hence the constellation $\Q$ is left cancellative.
\end{proof}

\section{Examples and non-examples of constellations with range}  \label{examples}

Aside from the first and last two examples to follow, all are ``concrete" in the sense that  the elements are functions between sets (perhaps with additional structure).  In each case, the standard order on $D(\Q)$ corresponds to the relation of being a substructure in the appropriate sense.  All except the last appeared in \cite{constgen} as examples of constellations.

\begin{eg} The constellation determined by a quasiordered set.   \label{eg:QU}  \end{eg}
Also introduced in \cite{constell} (see Example $2.4$ there), a constellation arises from any quasiordered set $(Q,\leq)$: simply define $e=e\cdot f$ whenever $e\leq f$ in $(Q,\leq)$, and let $D(e)=e$ for all $e\in Q$.  This constellation is easily seen to be normal if and only if the quasiorder is a partial order, and in that case the constellation $Q$ has range given by $R(e)=e$ for all $e\in Q$, and indeed is left cancellative.
 
\begin{eg} The constellation of sets. \label{eg:CSETR}
\end{eg}
Let $S$ be the class of sets.  There is a familiar category structure $\mathit{SET}$
associated with $S$, consisting of sets as the objects and maps between them as the arrows.  Taking the ``arrow only'' point of view, the category consists of all possible maps between all possible sets, with the  operations $D$ and $R$ given by specifying $D(f)$ to be the identity map on the domain of $f$ and $R(f)$ the identity map on its codomain, with the partial operation of category composition defined if and only if domains and codomains coincide.  Note that for $f,g\in \mathit{SET}$, it is possible for $D(f)=D(g)$ and $xf=xg$ for all $x\in \Dom(f)=\Dom(g)$ yet $R(f)\neq R(g)$.  

In \cite{constgen}, the constellation $\mathit{CSET}$ is defined from $\mathit{SET}$ by analogy with $\C_X$ as above, by taking the elements to be the \emph{surjective} functions, with $D$ defined as in the category, but with composition of functions $f\cdot g$ defined if and only if $\Img(f)\subseteq\Dom(g)$.  This normal constellation has range, with $R(f)$ the identity map on the image of $f$, easily seen to satisfy the congruence condition.  Note that $D(s\cdot t)=D(s)$, but $R(s\cdot t)$ is not simply $R(t)$, so there is no left/right symmetry as in the category.  This example is also left cancellative.  In contrast to $\mathit{SET}$, if $f,g\in \mathit{CSET}$ are such that $D(f)=D(g)$ and $xf=xg$ for all $x\in \Dom(f)=\Dom(g)$, then $f=g$ and so $R(f)=R(g)$.  

\begin{eg} The constellation of groups. \label{eg:CGRPR} \end{eg}

Let $G$ be the class of groups.  In the constellation $CGRP$, the elements are the surjective homomorphisms between groups, and composition and $D$ are as in the constellation of sets $\mathit{CSET}$.  The constellation $CGRP$ has range that satisfies the congruence condition, as in $CSET$, and is also left cancellative.  The example generalises widely: any class of algebras of the same type, such as rings, modules, semigroups and so on, will give rise to a generally left cancellative constellation with range in a similar way.  The left cancellative law will hold if morphisms are functions and $R$ is as in $\mathit{CSET}$, but does not hold if morphisms are binary relations, for example.

\begin{eg}The constellation of rings with additive homomorphisms. \label{eg:CRING+R} \end{eg}

Let $RING^+$ be the category of rings but with arrows being additive abelian group homomorphisms.  (That this is a category is easily checked.)  Again there is an associated constellation $CRING^+$, consisting of mappings of the form $f:R\rightarrow S$ where $S$ is ``as small as possible", which here means generated as a ring by Im$(f)$.  This example has pre-range: for each element $f:R\rightarrow S$, $R(f)=1_S$, the identity map on $S$. 
However, it does not satisfy the congruence condition and so $RING^+$ is not a constellation with range. To see this, consider the maps $f: \Z[x]\rightarrow \Z[x]$ given by $f(a_0+a_1x+\cdots+a_nx^n)=a_0+a_1x$ and $g: \Z[x]\rightarrow \Z[x^2]$ given by $g(a_0+a_1x+\cdots+a_nx^n)=a_0+a_2x^2$, where $\Z[x^2]$ is the ring of integer polynomials in which only even powers of $x$ appear.  Both $f,g$ are additive homomorphisms, Im$(f)$ generates $\Z[x]$ and Im$(g)$ generates $\Z[x^2]$, so both $f,g\in CRING^+$.  Now $fg$, the composite $f$ followed by $g$, maps $a_0+a_1x+\cdots+a_nx^n$ to $a_0$, and Im$(f)\subseteq \mbox{Dom}(g)$, so $f\cdot g:\Z[x]\rightarrow \Z$ is an element of $CRING^+$ and $R(f\cdot g)$ is the identity map on $\Z$.  But $R(f)$ is the identity map on $\Z[x]$, so $R(R(f)\cdot g)=R(g)$, which is the identity map on $\Z[x^2]$.
  
\begin{eg} The constellation of partial maps with infinite domain. \label{eg:CXinfR}\end{eg}  

Suppose $X$ is an infinite set, and denote by $\C^{\infty}_X$ the set of all partial maps in $\C_X$ that have infinite domains, a subconstellation of $\C_X$.  It was noted in \cite{constgen} that there is no analogous category, and indeed the normal constellation $\C_X^{\infty}$ does not have range: for if $s\in \C_X$ has finite image, then there is no smallest $e\in D(\C_X^{\infty})$ such that $s\cdot e$ exists.

\begin{eg}\label{eg:paconstTS} Constellations with range from actions of monoids on posets.  \end{eg}

This example is of quite a different flavour and comes from the connection between globalisation of partial actions of monoids,
categories, groupoids and, indeed, constellations with range,  and so-called maximum enlargement theorems for star-injective functors and radiants. {Here a radiant $\pi:\Q\rightarrow \mathcal{P}$ is {\em star-injective} if for any $s,t\in \Q$ with $s\pi=t\pi$ and $D(s)=D(t)$ we have that $s=t$.} The connections are well established for partial actions of {\em groups} \cite{kellendonk,lawsoninverse} and may be used to develop structure theorems for classes of monoids.   Here we give a taste of the theory as applied to {\em total} actions of {\em monoids} on partially ordered sets.

Suppose $X$ is a poset, and that the monoid $M$ acts monotonically on the right of $X$, so that for all $x,y\in X$ and $m,n\in M$,
\[ x\cdot 1=x, \,\,\, x\cdot(mn)=(x\cdot m)\cdot n,\mbox{ and } x\leq y \Rightarrow x\cdot m\leq y\cdot m.\]
Let $C(X,M)=X\times M$ equipped with the unary operation $D$ and the partial binary operation $\cdot\,$, given by
$$D((x,m))=(x,1)\mbox{ and } (x,m)\cdot (y,n)=(x,mn) \mbox{ providing $x\cdot m\leq y$}.$$
It is routine to check that $\Q=(C(X,M),\cdot\, ,D)$ is a (normal) constellation with range, and $R((x,m))=(x\cdot m,1)$ for all $(x,m)\in C(X,M)$.  

Now define $\theta: C(X,M)\rightarrow M$ by setting $(x,m)\theta=m$.  This is easily seen to be a star-injective radiant $\Q\rightarrow M$ (where we view $M$ as a single-object category, and hence as a constellation in which all products exist and $D(m)=1$ for all $m\in M$.  Moreover it satisfies the following condition: 
$$(m\in M\ \wedge\ e\in D(\Q)) \Rightarrow \exists m'\in \Q: (m'\theta=m \wedge D(m')=e). \hspace{2cm} (\alpha)$$  
In the language of categories (recalling that $M$ is a single-object category), this says that $\theta$ is a covering functor.

Conversely, if $\Q$ is a constellation with range, $M$ is a monoid and $\psi:\Q\rightarrow M$ is a star-injective radiant satisfying condition ($\alpha$), then it can be shown that $\Q\cong C(X,M)$, where $X=D(\Q)$ and the action of $M$ on $X$ is given by $e\circ m=R(m')$ where $m'\in \Q$ is such that $m'\psi=m$ and $D(m')=e$.  So, the action of a monoid $M$ on a poset $X$ corresponds to a constellation with range together with a suitable star-injective radiant.  

It is natural to ask what happens condition when ($\alpha$) is omitted in the constraints for $\psi$ above.  In this case, we must replace the action by a suitable partial action, and consider only pairs $(x,m)$ for which $x\cdot m$ exists. Again, we obtain a constellation with range.  Further details will appear in a later work.

\begin{eg}\label{eg:Ehr} Constellations with range from Ehresmann semigroups.  \end{eg}

Let S be an Ehresmann semigroup, so that S is a bi-unary semigroup with unary operations $D$ and $R$ (more frequently, $^+$ and $^*$ are used, as in \cite{lawson1} for example).  Defining a restricted product $\cdot$ on $S$ by the rule that $\exists a\cdot b$ if and only if $R(a)\leq D(b)$, in which case $a\cdot b=ab$, it is routine to check that $(S,\cdot,D,R)$ is a constellation with range.  In fact this process can be applied to any DR-semigroup in the sense of \cite{Dsemi} in which the domain and range operations satisfy the left and right congruence conditions, greatly increasing the supply of examples, including many in which the domain elements $D(s)$ need not form a semilattice nor even be closed under multiplication.

\section{Constellations with range are ordered categories with restriction}  \label{correspsec}

A strategy for many types of semigroup with domain-like and range-like operations is to try to represent them very economically, using a partially defined multiplication in which as few products as possible are retained, typically giving a small category, and using natural orders to recover the `lost' information.  As briefly discussed above, this was the approach of the ESN theorem for inverse semigroups, which led to the work of Lawson, who established an equivalence between Ehresmann semigroups and inductive categories in \cite{lawson1}.  It was also the approach used by Gould and Hollings in \cite{constell}, where they showed that left restriction semigroups correspond to what they called inductive constellations.  

We adopt a similar approach to constellations with range, using the derived category notion.  For the derived category to remain capable of capturing the information present in the original constellation with range, it is sufficient that it retain information about constellation products of the form $e\cdot s$, where $e\in D(\Q), s\in \Q$ and $e\leq D(s)$, since $s\cdot t=s\circ(R(s)\cdot t)$.  This is because if $s\cdot t$ exists (or equivalently by Lemma \ref{Rst}), $R(s)\cdot t$ exists, then $R(s)=D(R(s))=D(R(s)\cdot t)$, so $s\circ(R(s)\cdot t)$ exists and equals $s\cdot(R(s)\cdot t)=(s\cdot R(s))\cdot t=s\cdot t$.  To achieve this we equip the category with both a partial order and a (left) restriction operation.  We then show it is possible to reverse the process by beginning with such an ordered category with restriction and constructing from it a constellation with range.  The constructions are shown to be inverse, and a category isomorphism is established.

\begin{dfn}\label{defn:orderedconst}
We say $(\Q,\cdot\, ,D,\leq)$ is an {\em ordered constellation} if $(\Q,\cdot\, ,D)$ is a constellation and $\leq$ is a  partial order on $\Q$ (not necessarily the natural quasiorder) satisfying, for all $a,b,c,d\in \Q$:
\ben[label=\textup{(O\arabic*)}]
\item if $a\leq b$ and $c\leq d$ with both $a\cdot c$ and $b\cdot d$ existing, then $a\cdot c\leq b\cdot d$;
\item if $a\leq b$ then $D(a)\leq D(b)$;
\een
If the constellation with range $(\Q,\cdot\, ,D,R)$ is equipped with a partial order $\leq$ such that $(\Q,\cdot\, ,D,\leq)$ is an ordered constellation and for all $a,b\in \Q$, 
\ben[label=\textup{(O\arabic*')}]
\item[{\em (O$2^{\prime}$)}] if $a\leq b$ then $R(a)\leq R(b)$,
\een
then we say $(\Q,\cdot\, ,D,R,\leq)$ is an {\em ordered constellation with range}.  
\end{dfn}

Note that an ordered constellation is defined in \cite{constell} using rather stronger properties than those used here, insisting upon the existence of elements called {\em restrictions} and {\em corestrictions}, following  similar conventions for ordered groupoids. It is convenient here for us to disentangle the order from the existence of such elements, hence our Definition~\ref{defn:orderedconst}.   When the constellation happens to be a category, we recover the definition of an $\Omega$-structured category as in \cite{lawson1}, which is weaker than the definition of ordered category given there which also requires the following law:
\ben[label=\textup{(OO)}]
\item if $a\leq b$, $D(a)=D(b)$, and $R(a)=R(b)$ then $a=b$.
\een
Law (OO) holds in some but not all situations to follow, and in any case our definition of ordered constellation is also closer to the usual definition of ordered semigroups.

\begin{dfn}
Suppose $(\Q,\cdot\, ,D,\leq)$ is an ordered constellation.  We say it {\em has restrictions} if it satisfies the following:
\ben[label=\textup{(O3)}]
\item if $s\in \Q$ and $e\in D(\Q)$ with $e\leq D(s)$, then there exists $e|s\in \Q$, the unique $x\in \Q$ for which $x\leq s$ and $D(x)=e$.
\een
\end{dfn}

(O3) above is (OC8) (i) in Lemma $2.6$ appearing in \cite{lawson1}.

\begin{pro}  \label{conrest}
Let $\Q=(\Q,\cdot\, ,D)$ be a normal constellation, with $\leq$ the natural order on it.  Then $(\Q,\cdot\, ,D,\leq)$ is ordered and has restrictions given by $e|s = e \cdot s$ whenever $e \leq D(s)$.  Further, if $(\Q,\cdot\, ,D,R)$ is a constellation with range, then $(\Q,\cdot\, ,D,R,\leq)$ is an ordered constellation with range.
\end{pro}
\begin{proof}
By Proposition $3.3$ in \cite{Dsemiconst}, $(\Q,\cdot\, ,D,\leq)$ is quasiordered in which $e|s = e \cdot s$ whenever $e \leq D(s)$, and as $\Q$ is normal, $\leq$ is a partial order.   If $\Q$ has range then $\Q$ is normal by (R4) in Proposition \ref{Rlaws}, and if $a,b\in \Q$ are such that $a\leq b$ under the natural order, then $a=D(a)\cdot b$, so because $b\cdot R(b)$ exists, so does $(D(a)\cdot b)\cdot R(b)=a\cdot R(b)$.  Hence $R(a)\cdot R(b)$ exists and $R(a)\leq R(b)$.
\end{proof}

The case of ordered constellations of most interest to us here is that of categories.

\begin{dfn}\label{defn:fullycanc}
We say that $(Q,\circ\, ,D,R,\leq)$ is an {\em ordered category with restrictions} if $(Q,\circ\, ,D,R)$ is a category that is ordered and has restrictions as a constellation with range.
\end{dfn}

 Recall that in the category case, laws (O1), (O2) and (O$2^{\prime}$) define $\Omega$-structured categories, considered in Section 2 of \cite{lawson1} following their introduction by Ehresmann (see \cite{Ehres}), while (O3) corresponds to (OC8) (i) in \cite{lawson1}.  So it follows from what was observed by Lawson in Lemma $2.6$ of \cite{lawson1} that ordered categories with restrictions as defined here are indeed ordered categories with restrictions in the sense of \cite{lawson1}, since they are easily seen to satisfy the implication $$D(x)=D(y),\ x\leq y\ \Rightarrow\ x=y$$ 
(see the proof of Proposition $2.8$ in \cite{lawson1}), and hence satisfy law (OO).
 
The partial order on an ordered category with restrictions $(Q,\circ\, ,D,R,\leq)$ can be expressed purely in terms of restriction
and the partial order on $D(Q)$.  This is noted in \cite{lawson1} for the similar cases considered there, but we include the easy proof here for completeness.

\begin{pro}  \label{qorder}
Let $(Q,\circ\, ,D,R,\leq)$ be an ordered category with restrictions.
For $s,t\in Q$, $s\leq t$ if and only if $D(s)\leq D(t)$ and $s=D(s)|t$.
\end{pro}
\begin{proof}
If $s\leq t$ then $D(s)\leq D(t)$, so $D(s)|t$ exists and is the unique $x\leq t$ with $D(x)=D(s)$; thus $x=s$.  The converse is immediate.
\end{proof}

\begin{lem}  \label{catrest}
Let $(Q,\circ\, ,D,R,\leq)$ be an ordered category with restrictions.  Then for $e\in D(Q)$ and $s,t\in Q$ for which $e\leq D(s)$ and $R(s)\leq D(t)$:
\ben
\item[(1)]  $R(e|s)|t=R(e|s)|(R(s)|t)$ (and all exist), and
\item[(2)]  if $s\circ t$ exists then $e|(s\circ t)=(e|s)\circ R(e|s)|t$.
\een
\end{lem}
\begin{proof} Since $e|s\leq s$, $R(e|s)\leq R(s)\leq D(t)$.  Also, $R(s)|t$ exists and has domain $R(s)\geq R(e|s)$. Hence $R(e|s)|t$ and $R(e|s)|(R(s)|t)$ both exist.  Since both $R(e|s)|t$ and $R(e|s)|(R(s)|t)$ have domain $R(e|s)$ and because of the facts that $R(e|s)|(R(s)|t)\leq R(s)|t\leq t$ and $R(e|s)|t\leq t$, by uniqueness we have $R(e|s)|(R(s)|t)=R(e|s)|t$.

Suppose $s\circ t$ exists, so that $R(s)=D(t)$.  Now $R(e|s)|t$ exists and $R(e|s)|t\leq t$, so by (O1), $(e|s)\circ R(e|s)|t\leq s\circ t$.  Moreover, $D((e|s)\circ R(e|s)|t)=D(e|s)=e$, so by (O3), $e|(s\circ t)=(e|s)\circ R(e|s)|t$.
\end{proof}

We now obtain a characterisation of those categories that arise from constellations with range as in Proposition \ref{catran}.

\begin{thm}  \label{corresp}
If $(\Q,\cdot\, ,D,R)$ is a (left cancellative) constellation with range, then the derived category $(\Q,\circ\, ,D,R)$ is an ordered (left cancellative) category with restrictions in which $\leq$ is the natural order on $\Q$ as a normal constellation, and for $e\leq D(s)$, $e|s=e\cdot s$.

Conversely, if $(\C,\circ\, ,D,R,\leq)$ is a (left cancellative) ordered category with restrictions, then setting $s\cdot t$ equal to $s\circ (R(s)|t)$ whenever $R(s)\leq D(t)$ makes $(\C,\cdot\, ,D,R)$ into a (left cancellative) constellation with range, and the given partial order is nothing but the natural order on the constellation.  
\end{thm}
\begin{proof} 
Suppose $(\Q,\cdot\, ,D,R)$ is a constellation with range and define $s\circ t$ as before Proposition~\ref{catran}.  Proposition \ref{oneway} shows that $(\Q,\circ\, ,D,R)$ is a category, moreover one that is left cancellative as a category if $(\Q,\cdot\, ,D,R)$ is left cancellative.  By Proposition \ref{conrest}, $(\Q,\cdot\, ,D,R,\leq)$ is an
ordered constellation with range if $\leq$ is the natural order, moreover, $e|s$ exists as described.  Because $\circ$ is just a restricted form of $\cdot$, it follows immediately that $(\Q,\circ,D,R,\leq)$ is an ordered category with restrictions.  

Conversely, suppose $(\C,\circ,D,R,\leq)$ is an ordered category with restrictions, so that (C1)--(C3) hold (for $\circ$), $D(\C)=\{D(s)\mid s\in \C\}$ is the set of identities (equivalently right identities) of $(\C,\circ)$, and all of (O1), (O2), (O2$^{\prime}$) and (O3) hold for $\circ$.  Define the new product $s\cdot t$ on $C$ as in the theorem statement; this is well-defined since $D(R(s)|t)=R(s)$. Notice that where $x\cdot y$ exists then $D(x)=D(x\cdot y)$. Suppose $x,y,z\in \C$ are such that $x\cdot (y\cdot z)$ exists.  Then $R(y)\leq D(z)$, $y\cdot z=y\circ (R(y)|z)$, $R(x)\leq D(y\cdot z)=D(y)$ and $x\cdot(y\cdot z)=x\circ (R(x)|(y\circ(R(y)|z)))$.  Since $R(x)\leq D(y)$, $x\cdot y$ exists, and equals $x\circ (R(x)|y)$.  Also, $R(x\cdot y)=R(R(x)|y)\leq R(y)\leq D(z)$, so $(x\cdot y)\cdot z$ exists and $R(y)|z$ exists.  Hence
\bea
(x\cdot y)\cdot z&=&(x\circ (R(x)|y))\circ (R(R(x)|y)|z)\\
&=&x\circ ((R(x)|y)\circ R(R(x)|y)|z))\\
&=&x\circ ((R(x)|y)\circ R(R(x)|y)|(R(y)|z))\mbox{ by Lemma \ref{catrest} (1)}\\
&=&x\circ (R(x)|(y\circ (R(y)|z))\mbox{ by Lemma \ref{catrest} (2)}\\
&=&x\cdot(y\cdot z).
\eea  
So (Q1) holds.

Now suppose $x\cdot y$ and $y\cdot z$ exist.  Then $R(x)\leq D(y)=D(y\cdot z)$, so $x\cdot(y\cdot z)$ exists, establishing (Q2).

Pick $x\in \C$.  Then $R(D(x))=D(x)$, so $D(x)\cdot x=D(x)\circ (D(x)|x)=D(x)\circ x=x$.  If also $e\cdot x=x$ for some $e\in D(\C)$, then $D(x)=D(e\cdot x)=D(e)=e$, so (Q3) holds. Hence $(\C,\cdot\, ,D)$ is a constellation.

Supppose $e,f\in D(\C)$, and $e\cdot f,f\cdot e$ both exist.  Then $e=R(e)\leq f$ under the given partial order on $\C$ and similarly $f\leq e$, so $e=f$.  Hence $(\C,\cdot\, ,D)$ is normal.

Turning to range, Propositions \ref{Rlaws} and \ref{Rmin} between them show that establishing (R1)--(R4) for $R$ as given suffices to establish that $(\C,\cdot\, ,D,R)$ is a constellation with pre-range.  Law (R1) is immediate.  For all $s\in \C$,  $R(s)=D(R(s))$, so $s\cdot R(s)$ exists, establishing (R2). Also, if $s\cdot t$ exists then $R(s)\leq D(t)$, and so $R(R(s))\leq D(t)$, so $R(s)\cdot t$ exists, which establishes (R3).   (R4) is normality, shown above.  So $(\C,\cdot\, ,D,R)$ is a constellation with pre-range.  Finally, if $s\cdot t$ exists then $R(s\cdot t)=R(s\circ (R(s)|t))=R(R(s)|t)=R(R(s)\circ (R(s)|t))=R(R(s)\cdot t)$, so $(\C,\cdot\, ,D,R)$ satisfies the congruence condition.  

If $s\leq t$ under the given partial order, then $s=D(s)|t$ by Proposition \ref{qorder}, so $D(s)\cdot t=D(s)\circ(D(s)|t)=D(s)\circ s=s$.  Conversely, if $s=e\cdot t$ for some $e\in D(\Q)$, then $s=e\circ(e|t)=e|t\leq t$.  So the given partial order is the natural order on $(\C,\cdot\, ,D,R)$.

Suppose $(\C,\circ,D,R)$ is left cancellative.  If $s\cdot t=s\cdot u$ then $s\circ (R(s)|t)=s\circ (R(s)|u)$, so $R(s)|t=R(s)|u$, 
since $(\C,\circ,D,R)$ is left cancellative, and it follows that $R(s)\cdot t=R(s)\cdot u$.  So $(\C,\cdot\, ,D,R)$ is left cancellative.
\end{proof}

The constructions in the previous theorem are mutually inverse, as we now show.  If $(\Q,\cdot\, ,D,R)$ is a constellation with range and 
$(\C,\circ ,D,R,\leq)$ is an ordered category with restrictions, denote by ${\bf C}(\Q)$ the ordered category with restrictions $(\Q,\circ,D,R,\leq)$ obtained as in Theorem \ref{corresp} from $\Q$, and denote by ${\bf Q}(\C)$ the constellation with range $(\C,\circ,D,R)$ obtained as in that result from the ordered category with restrictions ${\mc C}$.  

\begin{pro}  \label{1to1}
Let ${\mc C}$ be an ordered category with restrictions and ${\mc Q}$ a constellation with range.  Then 
${\bf C}({\bf Q}({\mc C}))={\mc C}$ and ${\bf Q}({\bf C}({\mc Q}))={\mc Q}$.
\end{pro}
\begin{proof}
We begin by noting that  $D$ and $R$ are unchanged in the passage from $\C$ to ${\bf Q}({\mc \C})$ to ${\bf C}({\bf Q}({\mc \C}))$, and from $\Q$ to ${\bf C}({\mc \Q})$ to ${\bf Q}({\bf C}({\mc \Q}))$.  Similarly, the partial order $\leq$ is unchanged.

The constellation operation on ${\bf Q}({\mc C})$ is given by $x\cdot y=x\circ (R(x)|y)$ providing $R(x)\leq D(y)$, so the category operation on ${\bf C}({\bf Q}(\C))$ is given by 
$x\odot y=x\circ (R(x)|y)$ providing $R(x)=D(y)$.  But $D(y)|y=y$, so $x\odot y=x\circ y$ when each is defined, which is exactly when $R(x)=D(y)$.   
Restriction is of course determined by $\leq$ and $D$.

Conversely, the category operation on ${\bf C}({\mc Q})$ is $x\circ y=x\cdot y$ providing $R(x)=D(y)$.  Then the constellation operation on ${\bf Q}({\bf C}({\mc Q}))$ is $x:y=x\circ (R(x)|y)$ and is defined exactly when $R(x)\leq D(y)$, or, equivalently, if $x\cdot y$ is defined. In this case, $x:y=x\circ (R(x)| y)=x\cdot (R(x)\cdot y)=(x\cdot R(x))\cdot y=x\cdot y$.
\end{proof}

It follows that every ordered category with restrictions may equivalently be viewed as a constellation with range and vice versa.  Moreover, the relevant notion of morphism does not depend on which viewpoint is taken, as we show next.

\begin{dfn}\label{defn:rangeradiant}
A {\em range radiant} is a radiant $\rho:\Q_1\rightarrow \Q_2$ between constellations with range, additionally satisfying $(R(s))\rho=R(s\rho)$ for all $s\in \Q_1$.
\end{dfn}

The class of constellations with range is a category if arrows are taken to be range radiants, as is easily seen.
Likewise, the class of ordered categories with restrictions is itself a category in which the arrows are functors $\rho$ that are order-preserving, meaning that $s\leq t$ implies $s\rho\leq t\rho$.  Here, we do have something to check.

\begin{thm}  \label{equiv}
The category of ordered categories with restrictions is isomorphic to the category of constellations with range.
\end{thm}
\begin{proof}
Let ${\mc Q}_1$ and ${\mc Q}_2$ be constellations with range, which may also be viewed from a different perspective, as in Theorem \ref{corresp}, as ordered categories with restrictions ${\mc C}_1$ and ${\mc C}_2$.

Suppose $\rho:\Q_1\rightarrow \Q_2$ is a range radiant.  The fact that $\rho$ is a functor when $\Q_1$ and $\Q_2$ are viewed as the categories ${\mc C}_1$ and ${\mc C}_2$ is almost immediate, and it is easy to see that any radiant preserves the natural order on constellations.  

Conversely, let $\rho:\C_1\rightarrow \C_2$ be an order-preserving functor.  Then $D(s)\rho=D(s\rho)$ and $R(s)\rho=R(s\rho)$ for all $s\in \C_1$.  For  $e\in D(\C_1)$ we  have $e\rho=D(e)\rho=D(e\rho)\in D(\C_2)$.  If $e\leq D(s)$ then $e\rho\leq D(s)\rho=D(s\rho)$, so $e\rho|s\rho$ exists.  Moreover since $e|s\leq s$, we have $(e|s)\rho\leq s\rho$, and so because $D((e|s)\rho)=(D(e|s))\rho=e\rho$, by uniqueness we must have $(e|s)\rho=e\rho|s\rho$. If $s,t\in \C_1$ with $s\cdot t$ existing, then $R(s)\leq D(t)$, so $R(s)\rho\leq D(t)\rho$, and so $R(s\rho)\leq D(t\rho)$, giving  $s\rho\cdot t\rho$ exists, and 
\[(s\cdot t)\rho=(s\circ (R(s)|t))\rho=s\rho\circ (R(s)|t)\rho=s\rho\circ R(s)\rho|t\rho=s\rho\circ R(s\rho)|t\rho=s\rho\cdot t\rho.\]  
So $\rho$ gives a range radiant ${\mc Q}_1\rightarrow {\mc Q}_2$.  
\end{proof}


We saw earlier that every normal constellation embeds in a constellation with range.  It follows from the above results that every normal constellation arises as a subreduct of an ordered category with restrictions.

\section{Canonical extensions of constellations with range}  \label{canonext}

We first relate canonical extensions and congruences to radiants.  The following definition was not given in \cite{constgen} but is useful here.

\begin{dfn}
Let $\C$ be a category with $\Q$ a constellation.  We say the radiant $\rho:\C\rightarrow \Q$ is {\em canonical} if it is full, surjective and for all $a,b\in \C$,
\bi
\item if $a\rho=b\rho$ then $D(a)=D(b)$, and
\item if $a\rho=b\rho$ and $R(a)=R(b)$ then $a=b$.
\ei
\end{dfn}

Recall that a constellation is composable if for all $a\in \Q$, there is $e\in D(\Q)$ for which $a\cdot e$ exists.   Trivially, every constellation with range is composable.

\begin{pro}  \label{canonrad}
Let $\Q$ be a composable constellation.  Then $\rho: C(\Q)\rightarrow \Q$ given by $(s,e)\rho=s$ is a canonical radiant.  Conversely, if $\C$ is a category and $\rho:\C\rightarrow \Q$ is a  canonical radiant, then $ker(\rho)$ is a canonical congruence, $\C/ker(\rho)\cong \Q$, and $\C\cong C(\Q)$ under the isomorphism in which $s\mapsto (s\rho,R(s)\rho)$ for all $s\in \C$.
\end{pro}
\begin{proof} By Result \ref{surad}, $\rho: C(\Q)\rightarrow \Q$ given by $(s,e)\rho=s$ is a full surjective radiant, and obviously satisfies the conditions of Definition \ref{canonrad} and hence is canonical.  For the converse, if $\C$ is a category and $\rho:\C\rightarrow \Q$ is a canonical radiant then $\C/\ker(\rho)\cong \Q$ by Result \ref{canonmap}, since it is full and surjective.  Moreover, using  \cite[Corollary 3.7]{constgen} $\ker(\rho)$ is easily seen to be a canonical congruence, and so again by Result \ref{canonmap}, $\C/\ker(\rho)\cong \Q$ under the claimed isomorphism, the rest following from Result \ref{factext}.  
\end{proof}

Although we are here mainly interested in canonical extensions of constellations with range, we next observe some useful properties of $C(\Q)$ where $\Q$ is simply a constellation.  

\begin{dfn}
If $\Q$ is a constellation, define 
\[C_I(\Q)=\{(e,f)\mid e,f\in D(\Q), e\cdot f\mbox{ exists}\}.\]
\end{dfn}

The following is immediate. 

\begin{pro}
Let $\Q$ be a constellation.  Then $C_I(\Q)$ is a subcategory of $C(\Q)$ containing $D(C(\Q))$.
\end{pro}

It is well-known how to convert a quasiordered set into a thin category (one in which there is at most one arrow between any two objects).  For notational convenience we introduce the following terminology for this process.

\begin{dfn}  \label{pocat}
Let $(Q,\leq)$ be a quasiordered set.  Let ${\mathcal C}_Q$ denote the category with objects the elements of $Q$ and such for all $x,y\in Q$, there is an arrow from $x$ to $y$ if and only if $x\leq y$, the {\em thin category determined by $Q$}.
\end{dfn}

When the constellation $(\Q,\cdot,D)$ arises from a quasiordered set $(\Q,\leq)$ as in Example \ref{eg:QU}, $C_I(\Q)=C(\Q)$ is nothing but the thin category associated with $(\Q,\leq)$.  But more importantly here, when $\Q$ is chosen to be {\em CSET} or most of the other examples, $(e,f)$ in $C_I(\Q)$ is interpreted as the identity map on the domain of $e$ but with co-domain the larger set $\Dom(f)$.  Hence the elements of $C_I(\Q)$ correspond to {\em insertions} (or inclusion morphisms): there is precisely one for each pair of objects for which the source is a subobject of the target.  

For any constellation $\Q$, the category $C(\Q)$ has $\Q$ as a canonical quotient.  However, if $\Q$ is a constellation with range, a natural copy of $\Q$ (with restricted multiplication) also arises as a {\em subcategory} of $C(\Q)$. 

\begin{dfn}
Let $\Q$ be a constellation with pre-range.  Define \[C_S(\Q)=\{(s,R(s))\mid s\in \Q\}.\]
\end{dfn}

When the concrete examples of most interest to us are represented as $C(\Q)$ with $\Q$ the constellation with range of surjective morphisms, the elements of $C_S(\Q)$ correspond to the subcategory of $C(\Q)$ consisting of surjective morphisms.  Indeed we have the following, which also provides further justification for the definition of constellations with range.

\begin{pro} \label{constIScat}
Let $\Q$ be a constellation with pre-range.  Then $D(C(\Q))\subseteq C_S(\Q)$, and $C_S(\Q)$ is a subcategory of $C(\Q)$ if and only if $\Q$ is a constellation with range.  In that case, $C_S(\Q)$ is isomorphic to the derived category obtained from $\Q$ as in Definition \ref{derived}, and hence is an ordered category with restrictions.   
\end{pro}
\begin{proof} It is immediate that $D(C(\Q))\subseteq C_S(\Q)$. Suppose the congruence condition holds on $\Q$.  Choose $(s,R(s)),(t,R(t))\in \C_S(\Q)$ and suppose $(s,R(s))\circ (t,R(t))$ exists. Then $R(s)=D(t)$, so $R(s\cdot t)=R(R(s)\cdot t)=R(t)$ and so $(s,R(s))\circ (t,R(t))=(s\cdot t,R(t))=(s\cdot t,R(s\cdot t))\in C_S(\Q)$. Thus $C_S(\Q)$  is closed under composition and is therefore a subcategory.

Conversely, suppose $C_S(\Q)$ is a subcategory, hence closed under composition.  Now for all $s,t\in \Q$ for which $s\cdot t$ (and hence $R(s)\cdot t$) exists, we have $(R(s)\cdot t,R(R(s)\cdot t))\in C_S(\Q)$. Since $D(R(s)\cdot t)=R(s)$ it follows that 
\bea
(s,R(s))\circ(R(s)\cdot t,R(R(s)\cdot t)&=&(s\cdot (R(s)\cdot t),R(R(s)\cdot t))\\
&=&((s\cdot R(s))\cdot t,R(R(s)\cdot t))\\
&=&(s\cdot t,R(R(s)\cdot t))\in C_S(\Q).
\eea
Hence $R(s\cdot t)=R(R(s)\cdot t)$ and so the congruence condition holds on $\Q$.  The final comment follows from Theorem~\ref{corresp}.
\end{proof}

\begin{pro}  \label{isIS}
If $\Q$ is a constellation with range, then for all $s\in C(\Q)$, there is a unique $s'\in C_S(\Q)$ and $i\in C_I(\Q)$ such that $s=s'\circ i$.  
\end{pro}
\begin{proof}
For $(s,e)\in C(\Q)$, we have $s\cdot e$ exists, so $R(s)\leq e$ and we may write $(s,e)=(s,R(s))\circ (R(s),e)$, with $(s,R(s))\in C_S(\Q)$ and $(R(s),e)\in C_I(\Q)$; if $(s,e)=(t,R(t))\circ (R(t),e)$ also then $s=t\cdot R(t)=t$, establishing uniqueness. 
\end{proof}  

In most familiar concrete categories, every morphism may be factorised as an epimorphism followed by a monomorphism, although not uniquely.  What we have here is different: every morphism may be factorised uniquely as a ``surjection" (an element of $C_S(\Q)$) followed by an insertion (an element of $C_I(\Q)$).  Thus, if $f:A\rightarrow B$ is a morphism, one can generally express it as the surjection $f':A\rightarrow \Img(f)$ where $\Img(f)$ is a subobject of $B$, composed with the insertion mapping $f'': \Img(f)\rightarrow B$, and moreover this representation is unique.  Every insertion is a monomorphism but not conversely; however, the relationship between surjections and epimorphisms is more complex.  We explore this further in Section \ref{apply}.

\section{Introducing IS-categories}   \label{IScat}

Based on the properties of $C(\Q)$ where $\Q$ is a constellation with range, we here define IS-categories, which have subcategories with properties modelling those of $C_I(\Q)$ and $C_S(\Q)$ within $C(\Q)$.  We first introduce the more general notion of an I-category.
The category $C(\Q)$ is an I-category for any constellation $\Q$, so I-categories may prove to have interest for their own sake.

\begin{dfn} \label{catinsertdef}
Let $\C$ be a category.  Then it is a {\em category with insertions}, or an {\em I-category}, if there is a subcategory 
$\I_\C$ such that:
\ben[label=\textup{(I\arabic*)}]
\item $D(\C)\subseteq \I_\C$ and for $e,f\in D(\C)$, there is at most one $i\in \I_\C$ such that either ($D(i)=e$ and $R(i)=f$) or ($D(i)=f$ and $R(i)=e$), and 
\item if $s\circ i\in \I_\C$ where $s\in \C$ and $i\in \I_\C$, then $s\in \I_\C$.
\een
For $i\in \I_\C$, if $D(i)=e, R(i)=f$, write $i=i_{e,f}$.
\end{dfn}

If $\C$ is an I-category, it is easy to see that defining $\leq$ on $\I_\C$ by setting $e\leq f$ if and only if there is $i_{e,f}\in \I_\C$ gives a partial order on $D(\C)$, and that $\I_\C$ is a copy of the thin category arising from this partial order.

Evidently, if $\Q$ is a normal constellation then $C(\Q)$ is an I-category in which $\I_\C=C_I(\Q)$: (I1) is evidently satisfied, and if $(s,e)\circ (e,f)\in C_I(\Q)$, then $s\in D(\Q)$, so $(s,e)\in \I_\C$, whence so is (I2).  

The set of partial functions on a non-empty set $X$ is partially ordered by inclusion.  Similarly, in a concrete category, a natural way to define a partial order on the morphisms is to say $s\leq t$ whenever the function $s$ is a subset of the function $t$, and the co-domain of $s$ (possibly bigger than its image) is contained in that of $t$.   For concrete categories of the form $C(\Q)$ where $\Q$ is the constellation of surjections, this would be saying that $(s,e)\leq (t,f)$ providing $s\leq t$ and $e\leq f$, which is to say that $s=D(s)\cdot t$ in $\Q$, and $e\leq f$ in $D(\Q)$.  We can express this by saying that $(s,e)\circ (e,f)=(D(s),D(t))\circ (t,f)$, which given the needed compatibilities of domains and co-domains, is evidently equivalent to saying that $s\circ i=j\circ t$ for some $i,j\in C_I(\Q)$.  This turns out to be a partial order, even in a general I-category.  


\begin{dfn}
Let $\C$ be an I-category.  Define the relation $\leq_I$, the {\em I-order} on $\C$,
as follows: for all $s,t\in \C$, $s\leq_I t$ if and only if there are $i,j\in \I_\C$ such that $s\circ i$ and $j\circ t$ exist and are equal. 
\end{dfn}
 
We remark that, in the above definition, if $s\circ i=j\circ t$ then $i= i_{R(s),R(t)}$ and $j=i_{D(s),D(t)}$, which must both exist in $\I_\C$.  Thus for $e,f\in D(\C)$ we have $e\leq_I f$ if and only if $i_{e,f}\in \I_\C$ exists, so the I-order restricted to $D(\C)$ is simply the partial order on $D(\C)$ described after Definition \ref{catinsertdef}.  This gives one part of the following result.
  
\begin{pro}  \label{Icatordered}
Let $\C$ be an I-category.  Then it is an ordered category with respect to its I-order, and for $e,f\in D(\C)$, $e\leq_I f$ if and only if $i_{e,f}\in \I_\C$ exists.  Moreover, for $s\in \C$ and $i\in \I_\C$, $s\leq_I i$ implies $s\in \I_\C$.  
\end{pro}
\begin{proof}
Let $s,t,u\in \C$.  First, note that $D(s)\circ s=s\circ R(s)$, so $s\leq_I s$.  Secondly, if $s\leq_I t$ and $t\leq_I s$ then there are $i_1,i_2,j_1,j_2\in \I_\C$ such that $s\circ i_1=j_1\circ t$ and $t\circ i_2=j_2\circ s$.  It follows that
$s\circ i_1\circ i_2=j_1\circ j_2\circ s$, so that 
 $D(s)=R(j_2)=D(j_1)$ and so by uniqueness,
$j_1\circ j_2=D(s)$; similarly, $i_1\circ i_2=R(s)$, $i_2\circ i_1=R(t)$ and $j_2\circ j_1=D(t)$. Making use of (I1) we now obtain that $i_1=i_2=R(s)=R(t)$ and $j_1=j_2=D(s)=D(t)$, and
 so $s=t$.  Finally, if $s\leq_I t$ and $t\leq_I u$, then there are $i_1,i_2,j_1,j_2\in \I_\C$ such that  $s\circ i_1=j_1\circ t$ and $t\circ i_2=j_2\circ u$, so $j_1\circ t\circ i_2$ exists and equals both $s\circ i_1\circ i_2$ and $j_1\circ j_2\circ u$.  But since $i_1\circ i_2,j_1\circ j_2\in \I_\C$, this shows that $s\leq_I u$. Hence $\leq_I$ is a partial order on $\C$.

We now show $\C$ is an ordered category. Suppose that $s,t,u,v\in \C$, with $s\leq_I t,u\leq_I v$, and $s\circ u,t\circ v$ exist.  Then there are $i_1,i_2,j_1,j_2\in \I_\C$ such that $s\circ i_1=j_1\circ t$ and $u\circ i_2=j_2\circ v$.  Then since $s\circ u$ and $t\circ v$ exist, $D(j_2)=D(u)=R(s)=D(i_1)$, and $R(j_2)=D(v)=R(t)=R(i_1)$, so by uniqueness, $i_1=j_2$ and so  the following all exist and
$$s\circ u\circ i_2=s\circ j_2\circ v=s\circ i_1\circ v=j_1\circ t\circ v, $$
giving that $s\circ u\leq_I t\circ v$.  So (O1) holds.  

Now if $s,t\in \C$ with $s\leq_I t$, then there are $i,j\in \I_\C$ for which $s\circ i=j\circ t$, so $D(s)=D(s\circ i)=D(j\circ t)=D(j)$, while $R(j)=D(t)$, so $D(s)=D(j)\leq_I R(j)=D(t)$.  Furthermore, $R(i)=R(s\circ i)=R(j\circ t)=R(t)$, and because $s\circ i$ exists, $R(s)=D(i)\leq_I R(i)=R(t)$  Hence (O2) and (O2$'$) hold.  Hence $\C$ is an ordered category under $\leq_I$.  Moreover if $s\in \C, i\in \I_\C$ and $s\leq_I i$ then there are $j,k\in \I_\C$ for which $s\circ j =k\circ i\in \I_\C$, so $s\in \I_\C$.
%
\end{proof}

\begin{lem}  \label{inI}
If $\C$ is an I-category with $s\in \C$, then $s\in \I_\C$ if and only if $D(s)\leq_I s$.
\end{lem}
\begin{proof}
If $i\in \I_\C$ then $D(i)\circ i=D(i)\circ i$, so since $D(i)\in \I_\C$, we have $D(i)\leq_I i$.  Conversely, if $D(s)\leq_I s$, then there are $i,j\in \I_\C$ for which $D(s)\circ i=j\circ s$. We have $D(s)=D(j)$ and  $R(j)=D(s)$, giving $j=D(s)$, and we then obtain $i=s$, so $s\in \I_\C$.
\end{proof}

A useful characterisation of order-preserving functors between I-categories is as follows.

\begin{pro}  \label{Ipres}
Let $\C_1,\C_2$ be I-categories, ordered by their I-orders as in Proposition \ref{Icatordered}. Let  $F:\C_1\rightarrow \C_2$ be a functor.  Then $F$ is order-preserving if and only $\I_{\C_1}F\subseteq \I_{\C_2}$.
\end{pro}
\begin{proof}
Suppose $F$ is order-preserving, and pick $i\in \I_{\C_1}$. Let $e=D(i)$, and so  $e\leq_I i$ by Lemma~\ref{inI}.  Then $eF\leq_I iF$, and $eF=D(i)F=D(iF)$, so $D(iF)\leq_I iF$.  By Lemma \ref{inI}, $iF\in \I_{\C_2}$.

Conversely, suppose $\I_{\C_1}F\subseteq \I_{\C_2}$.  Then if $s,t\in \C_1$ with $s\leq_I t$, there are $i,j\in \C_1$ such that $s\circ i=j\circ t$, so applying $F$ we obtain $(sF)\circ (iF)=(jF)\circ (tF)$.  Since $iF,jF\in \I_{\C_2}$, it follows that $sF\leq_I tF$ under the I-order. 
\end{proof}

As previously noted, if $\Q$ is a normal constellation, then $C(\Q)$ is an I-category.  However, these two kinds of objects do not correspond in a one-to-one manner.   In \cite{constgen}, the following example of a constellation is given.  Let $\Q=\{s,e,f,g\}$, in which $D(\Q)=\{e,f,g\}$ with $D(s)=g$ and $s\cdot e,s\cdot f$ both existing but no other products existing aside from those that must.  (This can be realised as a constellation of partial functions, as shown in \cite{constgen}.) Then $\C=C(\Q)=\{(s,e),(s,f),(e,e),(f,f),(g,g)\}$ is an I-category, in which $\I_\C=C_I(\Q)=D(\C)=\{(e,e),(f,f),(g,g)\}$.    However, $\C\cong C(\C)$ (viewing $\C$ as a constellation), and so $C(\C)\cong C(\Q)$ with the insertion notions corresponding under the isomorphism, yet $\C$ and $\Q$ are clearly not isomorphic.  So the correspondence is not one-to-one.  However, it is possible that every I-category is isomorphic to one of the form $C(\Q)$ with $\I_\C$ corresponding to $C_I(\Q)$. 

Abstracting the properties of $C(\Q)$ summed up in the I-category definition and Proposition \ref{isIS} leads to the following.

\begin{dfn}\label{dfn:IS}
Let $\C$ be a category.  Then it is a {\em category with insertions and surjections}, or an {\em IS-category}, if is has  subcategories $\I_\C$
and $\Su_\C$, each containing all of $D(\C)$, such that 
\ben[label=\textup{(IS\arabic*)}]
\item for $e,f\in D(\C)$, there is at most one $i\in \I_\C$ such that either ($D(i)=e$ and $R(i)=f$) or ($D(i)=f$ and $R(i)=e$), and 
\item for all $a\in \C$, there are unique $s_a\in \Su_\C$ and $i_a\in \I_\C$ such that $a=s_a\circ i_a$.  
\een
\end{dfn}

Condition (I1) in the definition of an I-category appears in Definition~\ref{dfn:IS} as (IS1), whereas (I2) is missing here  because it is redundant.  

\begin{pro}  \label{ISisI}
Every IS-category is an I-category, in which $\I_\C$ is chosen as for the IS-category.  
\end{pro}
\begin{proof}
It is only necessary to show that if $k=a\circ i\in \I_\C$ where
$i\in \I_\C$, then $a\in \I_\C$.  But since $a=s_a\circ i_a$, we have  $D(k)\circ k=k=a\circ i=s_a\circ (i_a\circ i)$.  By uniqueness, $D(k)=s_a$ and so $a=D(k)\circ i_a\in \I_\C$.
\end{proof}

From Proposition \ref{isIS}, we immediately obtain the following.

\begin{cor}  \label{extIS}
Let $\Q$ be a constellation with range.  Then $\C=C(\Q)$ is an IS-category in which $\I_\C=C_I(\Q)$ and $\Su_\C=C_S(\Q)$. Moreover, for $a=(s,e)\in \C(\Q)$ we have
$s_a=(s,R(s))$ and $i_a=(R(s),e)$.
\end{cor}

Some useful facts follow.  First, note that in any IS-category $\C$ we have $D(\C)\subseteq \Su_\C\cap \I_\C$.  In fact, the opposite inclusion is also true.

\begin{pro}  \label{DCint}
Let $\C$ be an IS-category.  Then $\Su_\C\cap \I_\C=D(\C)$, and $\I_\C$ consists of monomorphisms. 
\end{pro}
\begin{proof}
If $s\in \Su_\C\cap \I_\C$, then $s=s\circ R(s)=D(s)\circ s$ are both ways to decompose $s$ into an element of $\Su_\C$ followed by an element of $\I_\C$, and so it follows that $s\in D(\C)$.  Next, suppose $i\in \I_\C$, and that $a\circ i=b\circ i$ for some $a,b\in \C$, and so $s_a\circ i_a\circ i=s_b\circ i_b\circ i$.  But by uniqueness, $s_a=s_b$ and $i_a\circ i=i_b\circ i$, so $D(i_a)=D(i_b)$ and $R(i_a)=D(i)=R(i_b)$, and then by (IS1), we have $i_a=i_b$.  So $a=s_a\circ i_a=s_b\circ i_b=b$, and so $i$ is a monomorphism.
\end{proof}

A further useful fact is the following, similar in form to (I2).

\begin{lem}  \label{useful}
Let $\C$ be an IS-category.  For $a\in \C$ and $i\in \I_\C$, if $a\circ i\in \Su_\C$, then $a\in \Su_\C$ and $i=R(a)$.
\end{lem}
\begin{proof}
With $i,a$ as given, we have $a\circ i=(a\circ i)\circ R(i)=s_a\circ (i_a\circ i)$ so that by uniqueness, $a\circ i=s_a$ and $R(i)=i_a\circ i$. It follows that  $R(a)=D(i)= R(i_a)$ and $R(i)= D(i_a)$,  so $i=i_a=R(a)$.  Hence $a=s_a\in \Su_\C$.
\end{proof}

Every category of the form $C(\Q)$ where $\Q$ is a constellation with range is an IS-category, so all the usual concrete categories have subcategories of surjective morphisms which can be viewed as constellations with range, and are therefore IS-categories; the categories of sets, groups, rings, topological spaces, posets, and so on, are all IS-categories.

There is some similarity with (orthogonal) factorization systems, where one distinguishes two subcategories of a category (modelling the epimorphisms and monomorphisms in a concrete category in which every morphism may be written as a product of an epimorphism followed by a monomorphism).  The differences here are that the two subcategories in an IS-category only contain $D(\C)$ rather than all isomorphisms as in a factorization system.  We return to this topic in Section \ref{apply}.

\section{Basic category concepts for IS-categories}  \label{apply}

Next, we consider how some well-known concepts of category theory behave in an IS-category.  Each IS-category $\C$ has in-built notions of ``surjective morphism" (an element of $\Su_\C$) and ``subobject" (via the I-order on $D(\C)$), and it is of interest to compare these with the usual category-theoretic definitions.  Other concepts such as equalisers and the balanced property may also be considered in an IS-category, relating them to properties of $\Su_\C$ and $\I_\C$ in a way impossible in general categories and with immediate application to familiar concrete categories.

In a concrete category, every surjective morphism is an epimorphism as well.  In such concrete categories viewed as IS-categories, the surjective morphisms are precisely the members of $\Su_\C$.  So assuming that in an IS-category every morphism in $\Su_\C$ is an epimorphism can give stronger results which still apply to concrete categories.

\begin{dfn}
An IS-category $\C$ is {\em regular} if $\Su_\C$ consists of epimorphisms.
\end{dfn}

\subsection{Epimorphisms}

In the categories of sets and groups, the epimorphisms are exactly the surjections.  However, it is well-known that in some familiar concrete categories, this is not the case. For example, consider the category of associative rings.  Here, there are epimorphisms of rings that are not surjective morphisms, one example being the insertion map of the ring of integers ${\mathbb Z}$ into the field of rational numbers ${\mathbb Q}$.  However, whether a member of $\Su_\C$ is an epimorphism is independent of whether we view it as an element of $\Su_\C$ or of $\C$ itself.

\begin{pro}  \label{ISlc}
Suppose $\C$ is an IS-category.  Then $s\in \Su_\C$ is an epimorphism in the category $\Su_\C$ if and only if it is an epimorphism in $\C$.  
\end{pro}
\begin{proof}
%
Suppose $s$ is an epimorphism in the category $\Su_\C$, and suppose $s\circ t=s\circ u$ for some $t,u\in \C$.  Then we may write $t=s_t\circ i_t$ and $u=s_u\circ i_u$, so 
\[
(s\circ s_t)\circ i_t=s\circ (s_t\circ i_t)=s\circ (s_u\circ i_u)=(s\circ s_u)\circ i_u.
\]
Since $s\circ s_t,s\circ s_u\in \Su_\C$, uniqueness implies that $s\circ s_t=s\circ s_u$ and $i_t=i_u$. Now, since $s$ is an epimorphism in $\Su_\C$,  $s_t=s_u$ and so $t=u$. Hence $s$ is an epimorphism in $\C$.  The converse is clear.
\end{proof}

\begin{cor}  \label{reglc}
The IS-category $\C$ is regular if and only if $\Su_\C$ is left cancellative.
\end{cor}

 In a category $\C$, an element $a$ is said to be an {\em isomorphism} if there exists $b\in \C$ with $a\circ b=D(a)$ and $b\circ a=D(b)$.  Every isomorphism is an epimorphism.  In an IS-category, although not every epimorphism is necessarily in $\Su_\C$, at least every isomorphism is.

\begin{pro}  \label{isosur}
Let $\C$ be an IS-category.  If $a\in \C$ has a left inverse in $\C$ (meaning there is $b\in \C$ such that $b\circ a=D(b)$), then $a\in \Su_\C$.  In particular, every isomorphism $a\in \C$ is in $\Su_\C$.
\end{pro} 
\begin{proof}
Suppose $b\circ a=D(b)$.  Then $D(b)=s_b\circ i_b\circ s_a\circ i_a=s_b\circ s_c\circ i_c\circ i_a$ where $c=i_b\circ s_a$.  So $D(b)=(s_b\circ s_c)\circ (i_c\circ i_a)$, and since $s_b\circ s_c\in \Su_\C$ and $i_c\circ i_a\in \I_\C$, we have by uniqueness that both equal $D(b)$ and so under the I-order, $D(b)=D(i_c)\leq R(i_c)=D(i_a)\leq R(i_a)=R(D(b))=D(b)$, and so all are equal, giving $i_a=i_c=D(b)$, and so $a=s_a\in \Su_\C$.
\end{proof}

The next result shows that the example in the category of rings of a non-surjective epimorphism which was also an insertion map was inevitable: if $\C$ is an IS-category and there is $s\in \C\backslash \Su_\C$ that is an epimorphism, there will be some $i\in \I_\C\backslash D(\C)$ that is an epimorphism.
 
\begin{pro}  \label{ISlc2}
Suppose $\C$ is an IS-category.  Then $\Su_\C$ contains all the epimorphisms of $\C$ if and only if $\I_\C\backslash D(\C)$ contains no epimorphisms.
\end{pro}
\begin{proof}
If $\Su_\C$ contains the epimorphisms of $\C$ then by Proposition \ref{DCint}, the only epimorphisms in $\I_\C$ are members of $D(\C)$.  Conversely, suppose the only epimorphisms in $\I_\C$ are in $D(\C)$.  Suppose $a\in \C$ is an epimorphism.  If $i_a\circ t=i_a\circ u$ for some $t,u\in \C$, then $s_a\circ i_a\circ t=s_a\circ i_a\circ u$, so $a\circ t=a\circ u$, and so $t=u$ since $a$ is an epimorphism.  Hence 
$i_a$ is an epimorphism and hence in $D(\C)$, and so $a=s_a\in \Su_\C$.
\end{proof}

\begin{cor}  \label{suepic}  
If $\C$ is a regular IS-category, then $\Su_\C$ is precisely the epimorphisms of $\C$ if and only if there are no epimorphisms in $\I_\C\backslash D(\C)$.
\end{cor}

\subsection{Subobjects}

In a constellation, the notion of subobject may be thought of as ``built-in".  We may define $e\in D(\Q)$ to be a {\em constellation subobject} of $f\in D(\Q)$ if and only if $e\leq f$ ($e\cdot f$ exists), or equivalently in $C(\Q)$, there is an insertion with domain $(e,e)$ and range $(f,f)$.   In concrete categories, this notion agrees with the intuitive notion of subobject.  It follows easily that in a constellation with range, one can perform the composition $s\cdot t$ if and only if $R(s)$ is a subobject in this sense of $D(t)$.  

However, in a category, the notion of subobject must in general be defined in a rather indirect way, since there is no natural ordering on domain elements.  In a category $\C$, recall that for two monomorphisms $a,b$ having the same ranges, we say {\em $a$ factors through $b$} if there is $x\in \C$ such that $a=x\circ b$, and we write $a\lesssim b$ in this case.  (Of course, if $a=x\circ b$ then $R(a)=R(b)$ so we do not need to assume this when defining $a\lesssim b$.)  

In general, $\lesssim$ is a quasiorder; let $\sim$ denote the equivalence relation on the monomorphisms of $\C$ determined by $\lesssim$.  \footnote{This is simply Green's relation $\mathcal{L}$ in the the associated semigroup with zero obtained from the subcategory of monomorphisms of $\C$ by putting all undefined products to be zero.}  In particular, if $a\sim b$ then $R(a)=R(b)$, so the equivalence relation on monomorphisms determined by equality of ranges is coarser than $\sim$. 

 We can now state the usual definition of subobjects in a category $\C$.  The {\em subobjects} of $e\in D(\C)$ are defined to be the equivalence classes of monomorphisms of $\C$ under $\sim$ amongst those having range $e$.  (Then $\lesssim$ corresponds to the intuitive notion of inclusion on the subobjects of $e$.)
 
Now if $\C=C(\Q)$ where $\Q$ is a constellation, we may compare these two notions of subobject.  In particular, one might hope that two monomorphisms $(s,e),(t,e)$ represent the same subobject of $(e,e)$ in the category-theoretic sense (that is, $(s,e)\sim (t,e)$) if and only if their images are equal (that is, $(R(s),R(s))=(R(t),R(t))$); equivalently, 
\[
(s,e)\sim (t,e)\mbox{ if and only if } (R(s),e)=(R(t),e).
\] 
In the language of IS-categories, Corollary~\ref{extIS} tells us that the above condition takes the following form: for monomorphisms $a,b\in \C$, 
\[
a\sim b\mbox{ if and only if }i_a=i_b.
\]

One direction of the above statement holds in any IS-category $\C$. 

\begin{pro}  \label{asb}
Let $\C$ be an IS-category.  For monomorphisms $a,b\in \C$  for which $R(a)=R(b)$, if $a\sim b$ then $i_a=i_b$.
\end{pro}
\begin{proof}
Suppose we have monomorphisms $a,b$  for which $R(a)=R(b)$, with $a\sim b$.  Then $a=x\circ b$ and $b=y\circ a$ for some $x,y\in \C$, and so $D(a)\circ a=x\circ y\circ a$, so because $a$ is a monomorphism, $D(a)=x\circ y=D(x)$.
Similarly, $D(y)=y\circ x$.  So $x,y$ are isomorphisms and then by Proposition \ref{isosur}, $x,y\in \Su_\C$.  Hence 
\[
s_b\circ i_b=b=y\circ a=y\circ s_a\circ i_a,\] 
and so because $y\circ s_a\in \Su_\C$, we obtain $s_b=y\circ s_a$ and $i_b=i_a$ as required.
\end{proof}  

But, the converse need not hold. 

\begin{dfn}  \label{wfso}
Let $\C$ be an IS-category.  We say $\C$ has {\em well-founded subobjects} if for every pair of monomorphisms $a,b$  for which $R(a)=R(b)$, we have that $a\sim b$ if and only if $i_a=i_b$.
\end{dfn}

There are easy reformulations of the concept of well-founded subobjects.

\begin{pro}  \label{wfso2}
Let $\C$ be an IS-category.  The following are equivalent.
\ben
\item[(1)] For each monomorphism $a\in \C$, there is $i\in \I_\C$ for which $R(a)=R(i)$ and $a\sim i$.
\item[(2)] For each monomorphism $a\in \C$, $a\sim i_a$.
\item[(3)] $\C$ has well-founded subobjects.
\een
\end{pro}
\begin{proof}
Assume (1). Let $a\in \C$ be a monomorphism.  Then there is $i\in \I_\C$ for which $R(a)=R(i)$ and $a=x\circ i=s_x\circ i_x\circ i$ for some $x\in \C$.   But $a=s_a\circ i_a$, so by uniqueness $i_x\circ i=i_a$, and so $i_a\lesssim i$.  We also have $i=y\circ a=y\circ s_a\circ i_a$ for some $y\in \C$, so $i\lesssim i_a$.  Hence $i_a\sim i\sim a$.  So (2) holds.  $(2)\Rightarrow (1)$ is immediate, and so $(1)\Leftrightarrow (2)$.

Note that for $i,j\in \I_\C$, we have that $i,j$ are monomorphisms by Proposition \ref{DCint}, so if  $R(i)=R(j)$ and $i\sim j$ then $i=i_i=i_j=j$, using Proposition \ref{asb}. If (2) holds, then for $a,b$ monomorphisms  with $R(a)=R(b)$, $a\sim b$ if and only if $i_a\sim i_b$, if and only if $i_a=i_b$, so (3) holds.  But if (3) holds, then for any monomorphism $a$, because $i_{i_a}=i_a$, we obtain $i_a\sim a$, and so (2) holds.   Hence, $(2)\Leftrightarrow (3)$.
\end{proof}

\begin{pro}  \label{wfs}
Let $\C$ be an IS-category. Then $\C$ has well-founded subobjects if and only if every monomorphism of $\C$ that lies in $\Su_\C$ is an isomorphism.  In this case,  for monomorphisms $a,b$ for which $R(a)=R(b)$, we have $a\lesssim b$ if and only if $R(s_a)\leq R(s_b)$.
\end{pro}
\begin{proof}
First, suppose $\C$ has well-founded subobjects.  Pick a monomorphism $a\in \Su_\C$.  By Proposition \ref{wfso2}, $a\sim i_a$, so there are $x,y\in \C$ for which $a=x\circ i_a$ and $i_a=y\circ a$.  By Lemma \ref{useful}, because $x\circ i_a=a\in \Su_\C$, $i_a=R(a)$, and then  $y\circ a=R(a)$, and so $a=a\circ R(a)=a\circ (y\circ a)=(a\circ y)\circ a=D(a)\circ a$. Because $a$ is a monomorphism we obtain $a\circ y=D(a)$ and $y\circ a=R(a)=D(y)$, so $a$ is an isomorphism (with inverse $y$).

Conversely, suppose every monomorphism of $\C$ that lies in $\Su_\C$ is an isomorphism. Pick $e\in D(\C)$, and let $a\in \C$ be a monomorphism for which $R(a)=e$.  If $x,y\in \C$ are such that $x\circ s_a=y\circ s_a$, then $x\circ a=x\circ s_a\circ i_a=y\circ s_a\circ i_a=y\circ a$, so $x=y$, and so $s_a$ is a monomorphism that lies in $\Su_\C$, hence an isomorphism by assumption.  So there is $b\in \C$ for which $s_a\circ b=D(s_a), b\circ s_a=D(b)=R(s_a)$.    Then $b\circ a=b\circ s_a\circ i_a=R(s_a)\circ i_a=i_a$.  So $a\sim i_a$.  So $\C$ has well-founded subobjects by Proposition \ref{wfso2}.

Under these equivalent conditions, consider monomorphisms $a,b\in \C$ such that $R(a)=R(b)$.  If $a\lesssim b$ then $a=x\circ b$ for some $x\in \C$, so $s_a\circ i_a=x\circ s_b\circ i_b=s_c\circ i_c\circ i_b$ where $c=x\circ s_b$, and so by uniqueness, $i_a=i_c\circ i_b$, and so $R(s_a)=D(i_a)=D(i_c)\leq R(i_c)=D(i_b)=R(s_b)$.  Conversely, if $R(s_a)\leq R(s_b)$, letting $j=i_{R(s_a),R(s_b)}$, we have that $i_a=j\circ i_b$, and so by Proposition \ref{wfso2}, $a\sim i_a\lesssim i_b\sim b$ and so $a\lesssim b$.
\end{proof}

The final part of the previous proposition can be interpreted as saying that if the IS-category has well-founded subobjects, then the category-theoretic notion of subobject inclusion for subobjects of a given object also corresponds with the ``natural" notion of inclusion: in $C(\Q)$, where $\Q$ is a constellation with range, for monomorphisms having the same range $(s,e),(t,e)$, we have that  $(s,e)\lesssim (t,e)$ if and only if $R(s)\leq R(t)$ in $\Q$.

So we have the following meta-theorem: ``{\em the category-theoretic concept of `subobject' is the correct one in a given IS-category if and only if every monomorphism of $\C$ that lies in $\Su_\C$ is an isomorphism}".  In particular, in the case of a concrete category $\C$, in which $\Su_\C$ is nothing but the surjective morphisms, this says that every surjective monomorphism must be an isomorphism.  In many such concrete categories, ``monomorphism = injection", and then the condition becomes that ``every bijective morphism is an isomorphism", a condition that holds in the category of sets as well as categories coming from algebraic varieties, but not in other cases such as the categories of topological spaces or partially ordered sets.  Hence in these latter cases, the category definition of subobject does not coincide with the natural notion.

\subsection{Equalisers} 

A familiar concept in category theory is that of an equaliser. If $\C$ is a category with $a,b\in \C$ for which $D(a)=D(b)$ and $R(a)=R(b)$, we say  $c\in \C$ is an {\em equaliser} of $a,b$ if $c\circ a=c\circ b$, and if $v\circ a=v\circ b$, then there is unique $h$ such that $v=h\circ c$.  We say $\C$ is a {\em category with equalisers} if it is a category in which any two elements have an equaliser.

 In most concrete categories consisting of functions, for two functions $a,b$ having equal domains and codomains, their equaliser $c$ can be chosen to be the insertion map from the domain of agreement of $a,b$ into their common domain.  In fact it is always possible to choose the equaliser to be an insertion in a regular IS-category.

\begin{pro}  \label{ISequal}
Suppose $\C$ is an IS-category with equalisers.   The following are equivalent.
\ben
\item[(1)]  For all $a,b\in \C$, there exists $i\in \I_\C$ which is an equaliser of $a,b$.
\item[(2)] $\C$ is regular.
\een
\end{pro}
\begin{proof}  Assume (1) holds.  Pick $s\in \Su_\C$, and suppose $s\circ a=s\circ b$ for some $a,b\in \C$.  By assumption, there exists $i\in \I_\C$ which is an equaliser of $a,b$.  Then there is $h\in \C$ for which $s=h\circ i$, so by Lemma \ref{useful}, $h=s$ and $i=R(s)$.  But $i\circ a=i\circ b$, so $a=b$.  So (2) holds.

Conversely, suppose (2) holds.  Suppose $a,b$ have equaliser $c$.  Then $c\circ a=c\circ b$, so $s_c\circ i_c\circ a=s_c\circ i_c\circ b$.  Since $s_c\in \Su_\C$, it  is an epimorphism, so that $i_c\circ a=i_c\circ b$.  If $v\circ a=v\circ b$, then there is $h$ such that $v=h\circ c=h\circ s_c\circ i_c$, so letting $h'=h\circ s_c$, we have $v=h'\circ i_c$.  If also $v=k\circ i_c$ then because $i_c$ is a monomorphism (by Lemma~\ref{DCint}), $h'=k$, establishing uniqueness.  Hence $i_c\in \I_\C$ is an equaliser of $a,b$.  So (1) holds.
\end{proof}

\subsection{The balanced property}  

A category is said to be {\em balanced} if every bimorphism is an isomorphism; here a bimorphism is a monomorphism that is an epimorphism.  (Of course every isomorphism is always a bimorphism.)  For regular IS-categories, the balanced property is equivalent to the two desirable properties just considered: every epimorphism being in $\Su_\C$ and $\C$ having well-founded subobjects.

\begin{pro}  \label{bal}
Suppose $\C$ is a regular IS-category.  The following are equivalent:
\ben
\item[(1)] $\C$ is balanced;
\item[(2)] $\Su_\C$ is precisely the epimorphisms of $\C$, and $\C$ has well-founded subobjects.
\een
\end{pro}
\begin{proof}
Suppose $\C$ is balanced.  If $i\in \I_\C$ is an epimorphism, then since it is also a monomorphism by Proposition~\ref{DCint}, it is an isomorphism, and so there is $s\in \C$ such that $s\circ i=D(s)$ and $i\circ s=D(i)$, so $s\in \I_\C$ and it follows that $s=i=D(s)=D(i)$ and then $i\in D(\C)$.  So by Corollary \ref{suepic}, the epimorphisms of $\C$ are precisely $\Su_\C$.   If $s\in \C$ is a monomorphism in $\Su_\C$ then it is an epimorphism and hence a bimorphism, hence is an isomorphism.  So by Proposition \ref{wfs}, $\C$ has well-founded subobjects.  

Conversely, suppose $\C$ is such that every epimorphism is in $\Su_\C$ and $\C$ has well-founded subobjects.  Suppose $s$ is a bimorphism.  Then it is a monomorphism in $\Su_\C$ and hence is an isomorphism by Proposition \ref{wfs}. So $\C$ is balanced.
\end{proof} 

So a concrete IS-category is balanced if and only if every epimorphism is a surjection and category-theoretic subobjects correspond to actual subobjects.

\subsection{Factorization systems}

A category $\C$ has a {\em factorization system} $(\Su,\M)$ if 
\ben
\item $\Su$ and $\M$ both contain all isomorphisms of $\C$, and are closed under composition (hence are subcategories of $\C$);
\item every $a\in \C$ can be factored as $a=s\circ m$ for some $s\in \Su$ and $m\in \M$;
\item the factorization has the following functoriality property: if we have $a,b\in \C$, $s,t\in \Su$ and $m,n\in \M$ for which 
$a\circ s\circ m=t\circ n\circ b$, then there is a unique $c\in \C$ for which $a\circ s=t\circ c$ and $n\circ b=c\circ m$.
\een

Earlier we noted similarities and differences between factorization systems and the pair $(\Su_\C,\I_\C)$ in an IS-category $\C$.  However, the functoriality property of a factorization system holds automatically in an IS-category.

\begin{pro}  \label{factor}
Suppose $\C$ is an IS-category.  Then we have the following functoriality property: if $a,b\in \C$ are such that there are $s,t\in \Su_\C$ and $i,j\in \I_\C$ for which $a\circ s\circ i=t\circ j\circ b$, then there is a unique $c\in \C$ for which $a\circ s=t\circ c$ and $j\circ b=c\circ i$.
\end{pro}
\begin{proof}
Suppose $a,b\in \C$ are such that there are $s,t\in \Su_\C$ and $i,j\in \I_\C$ for which $a\circ s\circ i=t\circ j\circ b$.  Then $s_{a\circ s}\circ i_{a\circ s}\circ i=t\circ s_{j\circ b}\circ i_{j\circ b}$, and so by uniqueness, 
\[s_{a\circ s}=t\circ s_{j\circ b}\mbox{ and }i_{a\circ s}\circ i=i_{j\circ b},\] 
so in particular, $R(s_{j\circ b})=R(s_{a\circ s})=D(i_{a\circ s})$.  

If there is $c\in \C$ for which $a\circ s=t\circ c$ and $j\circ b=c\circ i$, then we must have $s_{j\circ b}\circ i_{j\circ b}=s_c\circ i_c\circ i$, so $s_c=s_{j\circ b}$ and $i_c\circ i=i_{j\circ b}$, and so $D(i_c)=D(i_{j\circ b})$ and $R(i_c)=D(i)$.  But then $i_c\circ i=i_{j\circ b}=i_{a\circ s}\circ i$, so because $i$ is a monomorphism by Proposition \ref{DCint}, we obtain $i_c=i_{a\circ s}$.  So necessarily, $c=s_{j\circ b}\circ i_{a\circ s}$.

Next, note that whether or not such $c$ exists, $R(s_{j\circ b})=D(i_{j\circ b})=D(i_{a\circ s})$, so we can define $c=s_{j\circ b}\circ i_{a\circ s}$, and then
\[ t\circ c=t\circ  s_{j\circ b}\circ i_{a\circ s} = s_{a\circ s}\circ i_{a\circ s}=a\circ s,\]
and
\[c\circ i=s_{j\circ b}\circ i_{a\circ s}\circ i = s_{j\circ b}\circ i_{j\circ b} = j\circ b,\]
as required.
\end{proof}

As far as we know, the following has not previously been noted.   (Note that 2 above is not needed in the argument.)

\begin{pro}  \label{factiso}
Suppose a category $\C$ has $(\Su,\M)$ as a factorization system.  Then $\Su\cap \M$ is precisely the set of isomorphisms of $\C$.
\end{pro}
\begin{proof}
Pick $m\in \Su\cap \M$.  Then $D(m)\circ D(m)\circ m=m\circ R(m)\circ R(m)$, and because $D(m),R(m)$ are isomorphisms, they are members of both $\Su$ and $\M$, and of course $m\in \Su$ an $m\in \M$.  Applying the functoriality property, there exists $c\in \C$ such that $D(m)\circ D(m)=m\circ c$, and $R(m)\circ R(m)=c\circ m$, that is, $m\circ c=D(m)$ and $c\circ m=R(m)=D(c)$.  So $m$ is an isomorphism.
\end{proof}

The converse holds in an IS-category in which we let $\Su=\Su_\C$ and $\M$ the monomorphisms of $\C$. 

\begin{pro}
Suppose $\C$ is an IS-category, with subcategory of monomorphisms $\M$.  Then $(\Su_\C,\M)$ is a factorization system if and only if every $m\in \Su_\C\cap \M$ is an isomorphism.
\end{pro}
\begin{proof}  
Supppose $\Su_\C\cap \M$ is precisely the set of isomorphisms of $\C$.  Then suppose we have $a,b\in \C$, $s,t\in \Su_\C$ and $m,n\in \M$ for which $a\circ s\circ m=t\circ n\circ b$.  Now $m=s_m\circ i_m$ and $n=s_n\circ i_n$, so this is equivalent to
$$=a\circ (s\circ s_m)\circ i_m=(t\circ s_n)\circ i_n\circ b,$$
where $s\circ s_m, t\circ s_n\in \Su_\C$.  

Now if $x\circ s_m=y\circ s_m$ for some $x,y\in \C$, then $x\circ m=y\circ m$ (on multiplying both sides by $i_m$), so $x=y$; hence $s_m\in \M$ and so it is an isomorphism with inverse $s'_m$; similarly for $s_n$.  

Now the following are equivalent for $c\in \C$:
\bi
\item  $a\circ s=t\circ c$ and $n\circ b=c\circ m$;
\item $a\circ s\circ s_m=t\circ s_n\circ s_n'\circ  c\circ s_m$ and $s_n'\circ s_n\circ i_n\circ b=s_n'\circ c\circ s_m\circ i_m$;
\item $a\circ (s\circ s_m)=(t\circ s_n)\circ c'$ and $i_n\circ b=c'\circ i_m$ where $c'=s_n'\circ c\circ s_m$ or equivalently $c=s_n\circ c'\circ s_m'$.
\ei
But by Proposition \ref{factor}, there is unique $c'$ satisfying the final pair of equations, hence unique $c$ satisfying the first pair.  

The converse follows from Proposition \ref{factiso}.
\end{proof}

From Proposition \ref{wfs}, we have the following.

\begin{cor}
Suppose $\C$ is an IS-category, with subcategory of monomorphisms $\M$.  Then $(\Su_\C,\M)$ is a factorization system if and only if $\C$ has well-founded subobjects.
\end{cor}

From Proposition \ref{bal}, we obtain the following.

\begin{cor}
If $\C$ is a balanced regular IS-category, then $(\Su,\M)$ is a factorization system for $\C$, where $\Su$ and $\M$ are the epimorphisms and monomorphisms of $\C$ respectively.
\end{cor}

\section{IS-categories are canonical extensions of constellations with range}  \label{corresp2}

We saw in the previous section that IS-categories provide a useful enhancement to the category concept, helping to explain the observed behaviour of familiar concrete categories.  Of course, the definition of IS-categories was inspired by the behaviour of $C(\Q)$ where $\Q$ is a constellation with range.  In fact this connection proves to be very tight. 

\begin{thm}  \label{ISisCQ}
If $\C$ is a (regular) IS-category, then it is an ordered category under its I-order, $\Su_\C$ is a (left cancellative) ordered category with restrictions, and viewing the latter as a (left cancellative) constellation with range $\Su'_\C$,  the mapping $\rho:\C\rightarrow \Su'_\C$ for which $a\rho=s_a$ is a canonical radiant.  Moreover, there is an isomorphism of categories $\psi:\C\rightarrow C(\Su'_\C)$ in which $\Su_\C\psi=C_S(\Su'_\C)$ and $\I_\C\psi=C_I(\Su'_\C)$.  
\end{thm}
\begin{proof}
We have already seen in Proposition \ref{Icatordered} that $\C$ is an ordered category under its I-order $\leq_I$, and therefore so is $\Su_\C$ under this I-order inherited from $\C$. 

We must show that (O3) holds. To this end, suppose $s\in \Su_\C$, $e\in D(\Su_\C)$ and $e\leq D(s)$, so $e\circ j=i\circ D(s)$ for some $i,j\in \I_\C$, giving that $b=i\circ s$ exists.  But $s_b\circ i_b=i\circ s$, so $s_b\leq_I s$, and $D(s_b)=D(i)=D(i\circ D(s))=D(i)=e$.  Suppose $u\in \Su_\C$ is such that $u\leq_I s$ and $D(u)=e$.  Then there are $i_1,j_1\in \I_\C$ for which $u\circ i_1=j_1\circ s$, which forces $D(j_1)=e=D(i)$.  But  also, $R(j_1)=D(s)=R(i)$, so $j_1=i$ by uniqueness. Then $u\circ i_1=b=s_b\circ i_b$ so $u=s_b,i_1=i_b$ by uniqueness.  So (O3) holds, and so $\Su_\C$ is an ordered category with restrictions, and hence may be viewed as a constellation with range $\Su'_\C$ as in Theorem \ref{corresp}.   If $\C$ is regular then $\Su_\C$ is left cancellative by definition, whence by Theorem~\ref{corresp} so is $\Su'_\C$ as a constellation with range.

Next we show $\rho$ defined above is a canonical radiant.  First, pick $a,b\in \C$ such that $a\circ b$ exists.  Then $a\rho=s_a$ and $b\rho=s_b$.  Now $R(a\rho)=R(s_a)=D(i_a)\leq R(i_a)=R(a)=D(b)=D(s_b)=D(b\rho)$, so $a\rho\cdot b\rho=s_a\cdot s_b$ exists.  Letting $c=i_a\circ s_b$, we have 
$$(a\circ b)\rho=(s_a\circ i_a\circ s_b\circ i_b)\rho=(s_a\circ s_c\circ i_c\circ i_b)\rho=s_a\circ s_c$$ 
by uniqueness.  Now $R(s_a)\leq D(s_b)$ and then $R(s_a)|s_b=v\in \Su_\C$ where $v\leq_I s_b$, so $v\circ j=i\circ s_b$ for some $i,j\in \I_\C$ where $D(v)=R(s_a)$.  Since $D(i)=R(s_a)=D(i_a)$ and $R(i)=D(s_b)=R(i_a)$, uniqueness implies that $i=i_a$.  
So $v\circ j=i_a\circ s_b=s_c\circ i_c$ and by uniqueness again, $v=s_c$ (and $j=i_c$), so that
$$a\rho\cdot b\rho=s_a\cdot s_b=s_a\circ (R(s_a)|s_b)=s_a\circ s_c=(a\circ b)\rho.$$
Also, for $e\in D(\C)$, since $e\circ e=e$, $e\rho=e$; hence $\rho$ separates projections and so certainly respects $D$ and indeed $R$. Thus $\rho$ is a radiant.

Next, we show that $\rho$ is canonical. Suppose $a,b\in \C$ are such that $(a\rho)\cdot (b\rho)$ exists and is in the image of $\rho$, and $a=(a\rho)\circ i_a$, $b=(b\rho)\circ i_b$.  Because $(a\rho)\cdot (b\rho)$ exists, $R(a\rho)\leq D(b\rho)=D(b)$ from Proposition $3.5$.  Given the partial order is the $I$-order, let  $k=i_{R(a\rho),D(b)}$ be the unique member of $\I_\C$ such that $D(k)=R(a\rho)$ and $R(k)=D(b)$ as in Definition \ref{catinsertdef}, and let $c=(a\rho)\circ k$.  Then $c\rho=a\rho$ and $c\circ b$ exists.  So $\rho$ is full.

Next note that if $x,y\in \C$ are such that $x\rho=y\rho$ then $x=(x\rho)\circ i_x,y=(x\rho)\circ i_y$.  Clearly then $D(x)=D(y)$ and $D(i_x)=D(i_y)$.  If also $R(x)=R(y)$ then $R(i_x)=R(i_y)$, so by uniqueness, $i_x=i_y$, and hence $x=y$.  Finally, $\rho$ is surjective since for $s\in \Su_\C$, $s=s\circ R(s)$ and so $s\rho=s$. This finishes the argument that $\rho$ is a canonical radiant.

By Proposition \ref{canonrad} and  the above observation that $\rho$ fixes members of $D(\C)$ we now have that $\C\cong C(\Su_\C)$ under the isomorphism $\psi: a\mapsto (a\rho,R(a)\rho)=(a\rho,R(a))$.
Moreover, for $s\in \Su_\C$, because $s=s\circ R(s)$ and $s\in \Su_\C$ with $R(s)\in \I_\C$, by uniqueness we have $s\rho=s$, and so $s\psi=(s,R(s))\in C_S(\Su'_\C)$, that is, $\Su_\C\psi=\C_S(\Su'_\C)$.  Similarly, for $i\in \I_\C$, because $i=D(i)\circ i$ and $D(i)\in \Su_\C$ and $i\in \I_\C$, by uniqueness we have that $i\rho=D(i)$, and so $i\psi=(D(i),R(i))\in C_I(\Su'_\C)$.  Hence $\I_\C\psi\subseteq C_I(\Su'_\C)$.  Conversely, if $(e,f)\in C_I(\Su'_\C)$ where $e,f\in D(\Su'_\C)$ with $e\leq_I f$ under the I-order inherited from $\C$, there is $i\in \I_\C$ such that $D(i)=e$ and $R(i)=f$, so from what was just shown, we have $i\psi=(D(i),R(i))=(e,f)$, and so in fact $\I_\C\psi=C_I(\Su'_\C)$.
\end{proof}  

It follows from Theorem \ref{ISisCQ} that working in an IS-category is equivalent to working in a category of the form $C(\Q)$ where $\Q$ is a constellation with range.  

We remark that the unique factorisation of every element of an IS-category $\C$ into an element of $\Su_{\C}$ followed by an element of $\I_\C$ implies that $\C$ is an (internal) Zappa-Sz\'{e}p product of the subcategories $\Su_{\C}$ and $\I_{\C}$ in the sense of \cite{brin}.  Hence there are suitable (partial) actions of $\Su_{\C}$ on $\I_{\C}$ and vice versa, determined by the fact that any product $i\circ s$ of $i\in \I_{\C}$ and $s\in \Su_{\C}$ may be written as $s'\circ i'$ for unique choices of $i'\in \I_{\C}$ and $s'\in \Su_{\C}$. It follows that a copy of $\C$ may be re-constructed as pairs $(s,i)$, $s\in \Su_{\C}$ and $i\in \I_{\C}$ where $R(s)=D(i)$ as in Lemma $3.9$ in \cite{brin}, and then multiplication of these pairs is determined by the mutual actions. Expressing things in terms of the canonical extension of a constellation with range $\Pc$, all of this could be couched directly in terms of $\Pc$ itself, showing how $C(\Pc)$ can equivalently be defined as an external Zappa-Sz\'{e}p product of $\Pc$ viewed as an ordered category with restrictions and the thin category determined by $D(\Pc)$.

We next  show that the categories of constellations with range and IS-categories are equivalent, providing morphisms are defined appropriately.  For  constellations with range the appropriate notion of morphism is  that of range radiant, as in Definition~\ref{defn:rangeradiant}; we call the resulting category $\R$.

\begin{dfn}
Suppose $\C_1,\C_2$ are IS-categories.  We say a functor $F:\C_1\rightarrow \C_2$ is an {\em IS-functor} if $\I_{\C_1}F\subseteq \I_{\C_2}$ and $\Su_{\C_1}F\subseteq \Su_{\C_2}$.
\end{dfn}

Any IS-functor is order-preserving with respect to the I-orders by Proposition \ref{Ipres}.  

IS-functors are rather common.  Any of the forgetful functors taking the familiar IS-categories of mathematics to the IS-category of sets are examples of IS-functors, as is immediate, and indeed the free functors left adjoint to these are also IS-functors.  

It is clear that the identity map on an IS-category is an IS-functor, and the composition of two IS-functors is an IS-functor. We call the resulting category of IS-categories $\IS$.

\begin{lem}  \label{op}
Suppose $\Pc,\Q$ are constellations with range, and $\rho: \Pc\rightarrow \Q$ is a range radiant.  Then the mapping $F_{\rho}: C( \Pc)\rightarrow C(\Q)$ given by $(s,e)\mapsto (s\rho,e\rho)$ is an IS-functor.
\end{lem}
\begin{proof} 
Let $\rho: \Pc\rightarrow \Q$ be a range radiant. By Proposition $3.3$ of \cite{constgen}, $F_{\rho}$ is a functor.  
For $(e,f)\in C_I( \Pc)$, we have $(e,f)F_{\rho}=(e\rho,f\rho)\in C_I(\Q)$, so $C_I( \Pc)F_{\rho}\subseteq C_I(\Q)$.  Further, for $(s,R(s))\in C_S( \Pc)$, we have $(s,R(s)) F_{\rho}=(s\rho,R(s)\rho)=(s\rho,R(s\rho))\in C_S(\Q)$.  So $F_{\rho}$ is an IS-functor.
\end{proof}

\begin{lem}  \label{op2}
Suppose $\C,\D$ are IS-categories, and $F: \C\rightarrow \D$ is an IS-functor.  Then the mapping $\rho_F: \Su'_{\C}\rightarrow \Su'_{\D}$ given by restricting $F$ to $\Su_{\C}\subseteq \C$ is a range radiant. 
\end{lem}
\begin{proof}
First, if $s\in \Su'_{\C}$ and $e\leq D(s)$ so that $e|s\leq s$ exists in the constellation with range $\Su'_{\C}$ (where the natural order on $\Su'_{\C}$ is the inherited I-order from $\C$ by construction), then $e|s\leq_I s$ and so $(e|s)F\leq_I sF$ since $F$ is order-preserving by Propositions \ref{ISisI} and \ref{Ipres}. Then $(e|s)F\leq sF$ under the natural order on $\Su'_{\C}$ as a constellation, and so because $D((e|s)F)=D(e|s)F=eF$, it must be that $(e|s)F=(eF)|(sF)$ in $\Su'_{\C}$.

Next note that $\rho_F$ is well-defined, because $\Su'_{\C}F\subseteq \Su'_{\D}$.  Suppose $s\cdot t$ exists for $s,t\in \Su'_{\C}$.  Then $R(s)|t$ exists and $(s\cdot t)\rho_F=(s\circ (R(s)|t))F=(sF)\circ (R(s)|t)F=(sF)\circ (R(s)F)|(tF)=(sF)\circ (R(sF))|(tF)=(sF)\cdot (tF)=(s\rho_F)\cdot(t\rho_F)$. Also, $\rho_F$ respects $D$ and $R$ because $F$ did.  So $\rho_F$ is a range radiant.
\end{proof}

\begin{thm}  \label{catequiv}
The mapping $\phi$ taking any (left cancellative) constellation with range $\Q$ to its (regular) canonical extension $C(\Q)$ and any range radiant $\rho: \Pc\rightarrow \Q$ to the IS-functor $F_{\rho}:C( \Pc)\rightarrow C(\Q)$ as in Lemma \ref{op} is a functor $\R\rightarrow \IS$.

The mapping $\psi$ taking any (regular) IS-category $\C$ to the (left cancellative) constellation with range $\Su'_{\C}$ and any IS-functor $F:\C\rightarrow \D$ to its restriction $\rho_F$ to $\Su'_{\C}\rightarrow \Su'_{\D}$ as in Lemma \ref{op2} is a functor $\IS\rightarrow \R$.  

Moreover there are natural isomorphisms $\eta: I_{\R}\rightarrow\phi\psi$ and $\tau: I_{\IS}\rightarrow \psi\phi$ such that for all $\Pc\in \R$, $\eta_{\Pc}: \Pc\rightarrow C_S(\Pc)$ (the latter viewed as a constellation with range) is given by $s\eta_{\Pc}=(s,R(s))$ for all $s\in \Pc$, and
$\tau_{\C}:\C\rightarrow C(\Su'_{\C})$ is given, for all $x\in \Pc$, by $x\tau_{\C}=(s,R(x))$ where $x=s\circ i$ for $s\in \Su_{\C},i\in \I_{\C}$.  

Hence $\R$ and $\IS$ are equivalent categories, with associated functors $\phi$ and $\psi$.  
\end{thm}
\begin{proof} It is routine to check that $\phi$ and $\psi$ are functors. Note that if $\Q$ is left cancellative, then so is its derived category by Proposition \ref{oneway}, and this is isomorphic to $C_S(\Q)$ by Proposition \ref{constIScat}, which is therefore also left cancellative, and so $C(\Q)$ is regular.
%
Also note that if $\C$ is regular, then $\Su_{\C}$ is left cancellative as a category, hence as a constellation with range $\Su'_{\C}$ by Theorem \ref{corresp}.  

It remains to show that $\eta,\tau$ are natural isomorphisms.  We begin with $\eta$.  First, it is straightforward to check that for $\Pc\in \R$, $\eta_{\Pc}$ is an isomorphism of constellations with range.  We prove naturality.  Suppose $\rho:\Pc\rightarrow \Q$ is a range radiant.  Then for $s\in \Pc$, 
 
$$s(\eta_{\Pc}((\rho)(\phi\psi)))=(s,R(s))((\rho)(\phi\psi))=(s\rho,R(s)\rho),$$ 
whereas
$$s(\rho\eta_{\Q})=(s\rho,R(s\rho))=(s\rho,R(s)\rho)= s(\eta_{\Pc}((\rho)(\phi\psi))).$$
Hence $\eta_{\Pc}((\rho)(\phi\psi))=(\rho I_{\R})\eta_{\Q}$, where $I_{\R}$ is the identity functor on $\R$.  So $\phi\psi$ is naturally isomorphic to $I_{\R}$.

Next we consider $\tau$.  Again, for $\C\in \IS$, that $\tau: \C\rightarrow C(\Su'_{\C})$ is an isomorphism is immediate from Theorem \ref{ISisCQ}.  For naturality, suppose $F:\C\rightarrow \D$ is an IS-functor.  Then for all $x\in \C$, since $x=s_x\circ i_x$, we have that $xF=s_xF\circ i_xF$ with $s_xF\in \Su'_{\D}$ and $i_xF\in \IS_{\D}$, and so
$$x(\tau_{\C}(F(\psi\phi)))=(s_x,R(x))(F(\psi\phi))=(s_xF,R(x)F),$$
whereas
$$x(F\tau_{\D})=(s_xF,R(xF))=(s_xF,R(x)F)=x\tau_{\C}(F(\psi\phi)).$$
So $\tau_{\C}(F(\psi\phi))=(F I_{\IS})\tau_{\D}$, where $I_{\IS}$ the identity functor on $\IS$.  So $\psi\phi$ is naturally isomorphic to $I_{\IS}$.
\end{proof}

So IS-functors between IS-categories $\C,\D$ are in essence fully determined by the induced range radiants between $\Su'_{\C}$ and $\Su'_{\D}$.

\section{Open questions}  \label{open}

As observed in Proposition \ref{CX}, the normal constellation $(\C_X,\cdot\, ,D)$ has range, with $R(f)$ equal to the identity map on the image of $f$, and indeed $(\C_X,\cdot\, ,D,R)$ is left cancellative as noted previously.  It would be of interest to axiomatize those small constellations with range that are embeddable within $\C_X$.  A finite axiomatization exists in the case of semigroups of partial functions equipped with $D$ and $R$ (see \cite{scheinDR}), the axioms having a similar form to those of left cancellative constellations with range; this result will therefore carry over to inductive constellations with suitable range.  It is tempting to conjecture that every left cancellative constellation with range is so embeddable.  

We do not know of a counterexample to the assertion that every I-category is an IS-category for the same $\I_\C$ and some choice of $\Su$. Nor do we know whether each I-category is isomorphic to one of the form $C(\Q)$ for some constellation $\Q$.

\section*{Acknowledgements}
We thank a very careful referee for some helpful comments on the presentation of our work.

\begin{tabular}{ll}
Victoria Gould&
 Tim Stokes (corresponding author)\\
Department of Mathematics&
Department of Mathematics\\
University of York&University of Waikato\\
York YO23 3LT, U.K.& Hamilton 3216, New Zealand\\
{\em  victoria.gould@york.ac.uk}&
{\em tim.stokes@waikato.ac.nz}
\end{tabular}

\end{document}